\documentclass[12pt,leqno]{amsart}
\usepackage{microtype}
\usepackage{amssymb,amsmath,amsthm}
\usepackage{mathtools}
\usepackage[normalem]{ulem}

\usepackage[bookmarksopen,bookmarksdepth=3,colorlinks,citecolor=red,pagebackref,hypertexnames=true]{hyperref}
\usepackage[msc-links,nobysame,non-sorted-cites,initials]{amsrefs}
\usepackage[inline]{enumitem}
\usepackage{yfonts}
\mathtoolsset{showonlyrefs}


\DeclareFontFamily{U}{rsfs}{\skewchar\font127 }
\DeclareFontShape{U}{rsfs}{m}{n}{%
   <-6.5> rsfs5
   <6.5-8> rsfs7
   <8-> rsfs10
}{}
\newcommand{\interior}[1]{%
  {\kern0pt#1}^{\mathrm{o}}%
}

\usepackage[T1]{fontenc}  
\usepackage[lining]{libertine}
\usepackage{cabin}
\usepackage[libertine]{newtxmath}

\def\N {{\mathbb{N}}}
\def\R {{\mathbb{R}}}
\def\Z {{\mathbb{Z}}}

\def\M  {\mathrm {M}}
\def\d {\mathrm{d}}

\DeclareMathOperator{\supp}{supp}
\def\ind{\mathbf {1}}

\numberwithin{equation}{section}
\numberwithin{figure}{section}

\newtheorem{theorem}{Theorem}[section]
\newtheorem{proposition}[theorem]{Proposition}
\newtheorem{corollary}[theorem]{Corollary}
\newtheorem{lemma}[theorem]{Lemma}

\theoremstyle{definition}
\newtheorem{definition}[theorem]{Definition}
\newtheorem{remark}[theorem]{Remark}

\makeatletter
\DeclareRobustCommand\widecheck[1]{{\mathpalette\@widecheck{#1}}}
\def\@widecheck#1#2{%
    \setbox\z@\hbox{\m@th$#1#2$}%
    \setbox\tw@\hbox{\m@th$#1%
       \widehat{%
          \vrule\@width\z@\@height\ht\z@
          \vrule\@height\z@\@width\wd\z@}$}%
    \dp\tw@-\ht\z@
    \@tempdima\ht\z@ \advance\@tempdima2\ht\tw@ \divide\@tempdima\thr@@
    \setbox\tw@\hbox{%
       \raise\@tempdima\hbox{\scalebox{1}[-1]{\lower\@tempdima\box
\tw@}}}%
    {\ooalign{\box\tw@ \cr \box\z@}}}
\makeatother

\author[O.~Bakas]{Odysseas Bakas}
\address[O.~Bakas]{Department of Mathematics, University of Patras, 26504 Patras, Greece}
\email{\href{mailto:obakas@upatras.gr}{\textnormal{obakas@upatras.gr}}}

\author[I.~Parissis]{Ioannis Parissis}
\address[I.~Parissis]{Departamento de Matem\'aticas, Universidad del Pa\'is Vasco, Aptdo. 644, 48080 Bilbao, Spain and Ikerbasque, Basque Foundation for Science, Bilbao, Spain}
\email{\href{mailto:ioannis.parissis@ehu.eus}{\textnormal{ioannis.parissis@ehu.eus}}}

\thanks{O. Bakas was partially supported by the funding programme ``MEDICUS'' of the University of Patras.}

\thanks{I. Parissis is partially supported by grant PID2024-156267NB-I00 funded by MICIU/AEI/10.13039/501100011033 and cofunded by the European Union, grant IT1615-22 of the Basque Government and IKERBASQUE}

\begin{document}


\title[Endpoint estimates for multiparameter multipliers]{Endpoint estimates for multiparameter multipliers of Marcinkiewicz type}

\begin{abstract} In this paper we prove sharp endpoint estimates for multiparameter Marcinkiewicz multiplier operators. More precisely, this result is a consequence of a more general theorem for multiparameter $\mathcal R_{2,n}$--multipliers, a class that contains all multipliers of bounded $\mathcal{V}_q(\R^{\otimes n})$--variation for $1\le q<2$. The class $\mathcal R_{2,n}$ is a multiparameter generalization, introduced in this paper, of the $\mathcal R_2$--multipliers of Coifman, Rubio de Francia, and Semmes. We show that $\mathcal R_{2,n}$--multiplier operators locally map $L\log^{{3(n-1)}/{2}+{1}/{2}}L$ into $L^{1,\infty}$, and that this estimate is best possible, extending the corresponding one-parameter result of Tao and Wright to arbitrarily many parameters. We also establish the sharp bound $O((p')^{{3n}/{2}})$ for the $L^p(\mathbb R^n)\to L^p(\mathbb R^n)$ operator norms of such multiplier operators as $p \to 1^+$. The proof of our $L\log^{{3(n-1)}/{2}+{1}/{2}}L$-to-$L^{1,\infty}$ result combines a vector-valued endpoint estimate for the multipliers with an implicit square function characterization of $L\log^{\sigma/2}L$, obtained via duality from the Chang--Wilson--Wolff inequality. The latter produces, at each iterative step, auxiliary proxy functions that are fed into an intermediate one-parameter vector-valued weak-$(1,1)$ estimate.
\end{abstract}

\maketitle

\section{Introduction} This paper is concerned with endpoint bounds for multiparameter multiplier operators of Marcinkiewicz type near $L^1$. Recall that, in the one-dimensional case, a bounded function $m : \R \to \mathbb{C}$ is said to be a Marcinkiewicz multiplier if $m $ is of bounded variation uniformly on Littlewood--Paley intervals in $\mathcal{L} \coloneqq \{ \pm (2^k, 2^{k+1}]:\, k \in \Z \}$. For $n=2$, a bounded function $m : \R^2 \to \mathbb{C}$ is said to be a two-parameter Marcinkiewicz multiplier if $m$ is $C^1$ in the interior of each two-parameter Littlewood--Paley rectangle of the form $R= I \times J$ with $I,J \in \mathcal{L}$, and there exists a constant $C>0$ such that
\[
 \sup_{\eta \in \R} \sup_{I \in \mathcal{L}} \int_I | \partial_1 m (\xi, \eta) | \, \d \xi + \sup_{\xi \in \R} \sup_{J \in \mathcal{L}} \int_J | \partial_2 m (\xi, \eta) | \, \d \eta + \sup_{I,J \in \mathcal{L}} \int_{I \times J} |\partial_1 \partial_2 m (\xi, \eta) | \, \d \xi \d \eta \leq C.
\]
For $n \in \N$ with $n>2$, one defines $n$-parameter Marcinkiewicz multipliers on $\R^n$ in an analogous way. 

It is well-known that multiparameter Marcinkiewicz multiplier operators are $L^p$-bounded for all $p \in (1, +\infty)$; see e.g. \cite{Duobook}*{Theorem 8.3}. In \cite{TW}, Tao and Wright proved that, in the one-dimensional case, Marcinkiewicz multiplier operators locally map $L \log^{1/2} L$ into $L^{1,\infty}$ and that this is sharp. The authors of \cite{TW} deduced their aforementioned endpoint result from a more general endpoint theorem involving the class of $\mathcal{R}_2$--multipliers, introduced by Coifman, Rubio de Francia, and Semmes in \cite{CRDFS}, together with an implicit square function characterization of $L \log^{1/2} L([0,1])$ that was also established in \cite{TW}.

Recently, in a joint work with Ciccone and Vitturi \cite{BCPV}, we obtained sharp endpoint bounds for Marcinkiewicz multiplier operators of any order $N \in \N$. These are multiplier operators whose associated symbols are of uniform bounded variation over the component intervals of $\R \setminus E_N$, where $E_N$ is a lacunary set of order $N$. More specifically, it was shown in \cite{BCPV} that if $m$ is a Marcinkiewicz multiplier of order $N$, then $T_m$ satisfies a global $L \log^{N/2} L$-to-$L^{1, \infty}$ bound. One of the main ingredients of the proof in \cite{BCPV} was an extension, to Orlicz spaces of the form $L \log^{N/2} L$, $N \in \N$, of an implicit square function characterization of $L \log^{1/2} L$ established in \cite{TW}. This was achieved by showing that the characterization is equivalent, via duality, to the Chang--Wilson--Wolff inequality \cite{CWW}, together with the observation that a suitable form of the latter inequality, for large $p$, is amenable to iteration.

Since the Chang--Wilson--Wolff inequality actually holds in weighted $L^2$-spaces, the approach of \cite{BCPV} that was briefly outlined above has the virtue of being extendable to the vector-valued setting. Hence, we introduce in this paper an appropriate notion of multiparameter $\mathcal{R}_2$--multipliers, a class that includes multiparameter Marcinkiewicz multipliers and whose definition is inspired by the work of Xu \cite{Xu}. This allows us to obtain vector-valued extensions of the one-dimensional results of \cite{TW} in a set-up that supports the same iterative argument. As a result, in this paper we establish endpoint bounds for multiparameter $\mathcal{R}_2$--multiplier operators, and hence for Marcinkiewicz multiplier operators. More specifically, the main result of this article is the following sharp endpoint estimate for multiparameter $\mathcal R_2$--multipliers; for the precise definition of the class $\mathcal R_{2,n}$, see Section~\ref{sec:R2n} below.

\begin{theorem}\label{thm:multiparam} Let $R\subset \R^n$ be a rectangle with sides parallel to the coordinate axes. Let $\vec f =\{f_\tau\}\in L\log^{\alpha_n}L(R;\ell^2 _\tau)$ and let $\vec T_{\vec m}$ be a vector-valued $\mathcal R_{2,n}$--multiplier operator with $\| \vec m \|_{\mathcal{R}_{2,n ; \tau}} =1$. For all $\alpha>0$, there holds
\[
\frac{1}{|R|}\left|\left\{ u \in R:\, \left \|\,\vec T_{\vec m} \vec f(u)\, \right\|_{\ell^2 _\tau}>\alpha \right\}\right|\lesssim_n\frac{1}{\alpha}\left\|\,\vec f\, \right\|_{L\log ^{\alpha_n} L(R;\ell^2 _\tau)},
\]
where $\alpha_n\coloneqq \frac{3(n-1)}{2}+\frac{1}{2}$. The implicit constant depends only on $n$.
\end{theorem}

The definitions of local Orlicz norms and corresponding spaces are presented in \S\ref{sec:Orlicz}, while Theorem~\ref{thm:multiparam} is proved in Section~\ref{sec:proofmain}. As the class of $n$-parameter Marcinkiewicz multipliers is included in the class of $n$-parameter $\mathcal{R}_2$--multipliers, a direct consequence of Theorem~\ref{thm:multiparam} is the following sharp endpoint result for multiparameter Marcinkiewicz multiplier operators.

\begin{corollary}\label{thm:Marc_cor} Let $m$ be an $n$-parameter Marcinkiewicz multiplier with $ \| m \|_{M (\mathbb{R}^{\otimes n})} =1$. If $R\subset \R^n$ is a rectangle with sides parallel to the coordinate axes, then for all $\alpha>0$ there holds
\[
\frac{1}{|R|}\left|\left\{ u \in R:\, \left|  T_m  f(u) \right| >\alpha \right\}\right|\lesssim_n\frac{1}{\alpha}\left\|  f \right\|_{L\log^{\alpha_n} L(R)},
\]
where $\alpha_n\coloneqq \frac{3(n-1)}{2}+\frac{1}{2}$ and the implicit constant depends only on $n$.
\end{corollary}

By taking tensor products of the example in \cite{TW}*{\S3.2}, one deduces that, for $n \in \N$, the exponent $\alpha_n$ in Theorem~\ref{thm:multiparam} and Corollary~\ref{thm:Marc_cor} cannot be improved. See also the proof of \cite{Bak2019}*{Proposition 6.1}.

We remark that by Theorem~\ref{thm:Marc_cor} and Tao's converse extrapolation theorem \cite{Tao}, one can obtain an estimate concerning the growth of the local $L^p-L^p$ operator norms of $n$-parameter Marcinkiewicz multipliers as $p \to 1^+$. In this paper, employing different techniques, we obtain the following global result concerning the optimal growth of the $L^p-L^p$ operator norms of vector-valued $\mathcal R_{2,n}$--multipliers as $p \to 1^+$.

\begin{theorem}\label{thm:Lp_growth} Let $\vec T_{\vec m}$ be a vector-valued $\mathcal R_{2,n}$--multiplier operator with $\| \vec m \|_{\mathcal{R}_{2,n ; \tau}} =1$. For $1<p<\infty$, there holds 
\[
\left\| \, \vec T_{\vec m} \vec f  \, \right\|_{L^p(\R^n;\ell^2_{\tau})} \lesssim_n \max(p,p')^{\beta_n} \left\|\,   \vec f \,  \right\|_{L^p(\R^n;\ell^2_{\tau})},
\]
where $\beta_n \coloneqq \frac{3n}{2}$. The implicit constant depends only on $n$.
\end{theorem}

The proof of Theorem~\ref{thm:Lp_growth} is given in Section~\ref{sec:Proof_growth}. Note that, in view of the fact that the class of $n$-parameter Marcinkiewicz multipliers is contained in the class of $\mathcal{R}_{2,n}$--multipliers, Theorem~\ref{thm:Lp_growth} immediately implies the following.

\begin{corollary}\label{thm:Marc_Lp} Let $m$ be an $n$-parameter Marcinkiewicz multiplier with $\| m \|_{M (\mathbb{R}^{\otimes n})} =1$. For $1<p<\infty$, there holds  
\[
\left\| T_m  f  \right\|_{L^p(\R^n )} \lesssim_n \max(p,p')^{\beta_n} \left\| f \right\|_{L^p(\R^n)},
\]
where $\beta_n \coloneqq \frac{3n}{2}$ and the implicit constant depends only on $n$.
\end{corollary}

The $L^p-L^p$ operator norm of the classical $n$-parameter Littlewood--Paley square function behaves like $(p')^{\beta_n}$ as $p \to 1^+$; see \cite{Bak2021}. A simple argument involving Khintchine's inequality then shows that the exponent $\beta_n$ in Corollary~\ref{thm:Marc_Lp} cannot be improved, and the same holds for Theorem~\ref{thm:Lp_growth}.

The proof of our main result, Theorem~\ref{thm:multiparam}, proceeds by induction on the number of parameters $n$, iterating vector-valued estimates for the multipliers in hand at each step. A key tool in this induction is Proposition~\ref{prp:vvaluedTW}, which bounds a sum of frequency-localized pieces of a vector-valued family of functions in $L^{1,\infty}$. The bound holds once each piece is majorized, after convolution with a suitable bump function, by an auxiliary, or \emph{proxy}, function. The critical point, and the main novelty of the proof, is that each step of the induction couples this proposition, applied in the variable being isolated at that stage, with the local $L^1$ bound that follows from the $(n-1)$-parameter inductive hypothesis, applied in the remaining variables. This coupling is essential because it allows the induction to iterate the majorant proxies $F_{\tau,I}$ of Proposition~\ref{prp:vvaluedTW}, rather than the functions $f_{\tau,I}$ themselves, which would otherwise lose half a logarithmic order of integrability along the way. Since these proxies must be constructed for functions that are themselves defined via the $(n-1)$-parameter operator arising from the previous step of the induction, their construction cannot be black-boxed; instead, at each step we build them by hand, via a suitable dual, vector-valued form of the Chang--Wilson--Wolff inequality in one of the parameters.

Besides Theorem~\ref{thm:multiparam} and its proof, the paper introduces the class $\mathcal R_{2,n}$ of multiparameter $\mathcal R_2$--multipliers, inspired by the work of Xu, \cite{Xu}. This class contains multiparameter Marcinkiewicz multipliers and is compatible with the multiparameter notion of bounded variation in \cite{Xu}. Its definition is also what makes the inductive, vector-valued argument described above possible. We also extend the implicit square function characterization of $L\log^{\sigma/2}L$ to the multiparameter, vector-valued setting required by the induction. Finally, in the proof of Theorem~\ref{thm:Lp_growth} we make use of the $p$-independent inverse Rubio de Francia-type square function estimate of Bourgain~\cite{bourg_sq}; see also Kislyakov and Parilov~\cite{KislPar}. While estimates of this type have appeared before in related contexts, we believe this specific application to multiparameter multipliers is new.
 
\subsection{Background and history} Marcinkiewicz multipliers are named after J\'ozef Marcinkiewicz, who introduced them in \cite{Marcinkiewicz} and, using the then-new theory of Littlewood and Paley \cite{LP1}, \cite{LP2}, \cite{LP3}, established their $L^p$-boundedness for $p\in(1,+\infty)$ in the periodic setting. See the survey of Carbery \cite{Carbery} for the later development of Littlewood--Paley theory and its central role in harmonic analysis.
 
In the one-dimensional case, a proper subclass of Marcinkiewicz multipliers that arises naturally in applications is that of H\"ormander--Mihlin multipliers, which are classical Calder\'on--Zygmund operators and hence of weak-type $(1,1)$. It was long known that some one-dimensional Marcinkiewicz multiplier operators fail weak-type $(1,1)$ (see e.g. \cite{Edwards_Gaudry}*{\S7.5}), but their precise behavior near $L^1$ remained open until the end of the previous century. Tao and Wright \cite{TW} proved that Marcinkiewicz multiplier operators locally map $L \log^{1/2} L$ into $L^{1, \infty}$, sharply, and map $H^1(\R)$ into $L^{1,\infty}(\R)$ --- the latter also obtained independently by Kislyakov \cite{KislMarc}. As noted in \cite{KislMarc}, this has no higher-dimensional analog: the two-parameter Littlewood--Paley square function does not map $H^1(\R\times\R)$ into $L^{1,\infty}(\R)$ \cite{Bak2019}, so some two-parameter Marcinkiewicz multipliers fail likewise.
 
The endpoint behavior of the proper subclass of $n$-parameter Marcinkiewicz multipliers that behave like $n$-fold tensor products of (sufficiently smooth) H\"ormander--Mihlin multipliers has been thoroughly studied over the years; for simplicity, we call this class $n$-parameter H\"ormander--Mihlin multipliers below, without specifying the smoothness assumptions. One of the earliest endpoint results in multiparameter harmonic analysis, the local $L \log L$-to-$L^{1,\infty}$ bound for the maximal double Hilbert transform, is due to Zygmund \cite{Zyg49}, using complex-analytic methods. Later real-variable proofs were given by C. Fefferman \cite{CFef72} and, via product Hardy space theory, by R. Fefferman and Stein \cite{RFef_St}. The same techniques were used in \cite{Rfef} to deduce that two-parameter H\"ormander--Mihlin multipliers locally map $L \log L$ into $L^{1, \infty}$; another proof, providing also multilinear extensions, is due to Workman \cite{Workman}. Using one-parameter Hardy space theory and a Sobolev representation formula, Wojciechowski \cite{Woj2000} showed that the $L^p$-operator norm of $n$-parameter H\"ormander--Mihlin multipliers behaves like $(p')^n$ as $p \to 1^+$; for $n=2$ this can also be recovered from weighted bounds for two-parameter Calder\'on--Zygmund operators, obtained via a wavelet representation by Di Plinio, Wick, and Williams \cite{DWW}. For $p \in (0,1]$, Carbery and Seeger \cite{Carb_Seeg} established $H^p$-to-$L^p$ bounds for $n$-parameter Calder\'on--Zygmund operators, $n \geq 2$ (the case $n=2$ having been treated earlier by R. Fefferman \cite{Rfef}). More recently, Cowling, Lee, Li, and Pipher \cite{CLLP} proved global $L \log L$-to-$L^{1,\infty}$ bounds for product singular integrals on stratified Lie groups, via an atomic decomposition of $L \log L$ built from level-set estimates for two-parameter maximal operators and area integrals, in the spirit of Brossard \cite{Brossard} and Merryfield \cite{Mer}.

Back to the one-dimensional case, the main result from \cite{BCPV} shows that Marcinkiewicz multiplier operators of order $N$ map $L \log^{N/2} L$ into $L^{1, \infty}$ via an implicit square function characterization of $L \log^{N/2} L$, which recovers Zygmund's inequality for $N=1$ \cite{Zygmund_book}*{Theorem 7.6, Chapter XII} and Bonami's inequality for $N>1$ \cite{Bonami}*{corollaire 4}. This connection between endpoint results for Marcinkiewicz-type multipliers and Zygmund-type inequalities, already exploited in \cite{TW}, was studied systematically in \cite{BCDFPV}, joint with Ciccone, Di Plinio, Fraccaroli, and Vitturi. Given an Orlicz space $L^2 ([0,1]) \subset X \subset L^1 ([0,1])$, we characterized in \cite{BCDFPV}, in terms of a scale-invariant version of Zygmund's inequality, the closed null subsets $E \subset \R$ for which a multiplier with singular set $E$, i.e., frequency singularities on $E$, satisfies an $X$-to-$L^{1, \infty}$ bound. Optimal distribution bounds for one-dimensional Marcinkiewicz multipliers were recently obtained by Sukochev, Yang, Zanin, and Zhou \cite{SYZZ}.

\subsection{Structure of the paper} In Section~\ref{sec:notation} we set up the notation that is used throughout the paper. In Section~\ref{sec:1-d_vv} we establish suitable one-parameter vector-valued endpoint estimates, stated below as Proposition~\ref{prp:vvaluedTW} and Theorem~\ref{thm:vvaluedR_2}, that will be used in the proof of our main result. In Section~\ref{sec:R2n} we introduce the class of $n$-parameter $\mathcal{R}_2$--multipliers and in Section~\ref{sec:proofmain} we prove Theorem~\ref{thm:multiparam}. The proof of Theorem~\ref{thm:Lp_growth} is given in Section~\ref{sec:Proof_growth}. 


\subsection*{Acknowledgements} Several problems in the first author's PhD thesis, supervised by Jim Wright, were motivated by questions closely related to those addressed in the present paper. The first author gratefully acknowledges Jim Wright for many valuable discussions on topics related to this work over the years.

\section{Notation and preliminaries}\label{sec:notation} Before turning to the main part of the paper, we briefly digress in order to fix some notation used throughout the paper: the relevant notions from Orlicz spaces, the Littlewood--Paley intervals and projections, and a few further conventions.

\subsection{Orlicz spaces}\label{sec:Orlicz}   Throughout the paper, for $\sigma \geq 0$, we use the Orlicz function $\Phi_{\sigma}$ given by
\[
\Phi_{\sigma} (t)\coloneqq t [\log(e+t)]^{\sigma}, \qquad t \geq 0.
\]
Note that $\Phi_{\sigma}$ is monotone increasing, convex and $\Phi_{\sigma} (0) = 0$. Thus, for every $t \geq 0$ and every $\theta \in [0,1]$,
\begin{equation}\label{eq:convex-scale}
\Phi_\sigma(\theta t)\le \theta \Phi_\sigma(t).
\end{equation}
In addition, one has that $\Phi_\sigma$ is submultiplicative, namely, for $a,u\geq 0$, there holds
\begin{equation}\label{eq:submult}
\Phi_{\sigma} (a u)\lesssim_\sigma \Phi_{\sigma} (a)\Phi_{\sigma} (u)
\end{equation}
with implicit constant depending only on $\sigma$; see \cite{UMP}*{\S5.2}.

For a measure space $(X, \mathcal{A}, \mu)$ we define $ L \log^{\sigma} L (X) $ to be the class of all measurable functions $f : X \to \mathbb{C}$ such that for some $\alpha>0$, and hence, in view of \eqref{eq:submult}, for all $\alpha>0$, $\Phi_{\sigma} (\alpha^{-1} |f|) \in L^1(X)$. For an introduction to Orlicz spaces appearing in this paper, we refer the reader to Chapter 10 in \cite{WilsonBook}.

In the case where $\mu (X) $ is finite, for $f \in L \log^{\sigma} L (X)$ we define its  Luxemburg norm  by
\[
\|f\|_{L\log^{\sigma} L(X)}\coloneqq \inf\left\{t>0:\, \frac{1}{\mu(X)}\int_{X} \Phi_{\sigma} \left(\frac{|f(x)|}{t}\right)\,\d \mu(x)\leq 1 \right\} .
\]
In what follows we will denote by $\mu_X$ the corresponding normalized measure so that
\[
\int_X f \,\d\mu_X=\frac{1}{\mu(X)}\int_X f \,\d \mu. 
\]
We will also use the following notation for averages: if $X=S$ is a set of finite measure in $\R^n$ and $\mu$ is $n$-dimensional Lebesgue measure then
\[
   \langle |f| \rangle_{\Phi_\sigma,S}\coloneqq \|f\|_{L\log^{\sigma} L(S)}.
\]
Using this notation, the Orlicz maximal functions are defined by means of 
\[
   \M_{\Phi_\sigma}f(x)\coloneqq \sup_{Q\ni x}  \langle |f| \rangle_{\Phi_\sigma,Q},
\]
the supremum being taken over axis-parallel cubes in $\R^n$ with $Q\ni x$. If dyadic cubes are used in the definition, we write $\M_{\mathcal D,\Phi_\sigma}$. For $\sigma=0$ we omit $\Phi_0$ from the notation and just write $\langle |f|\rangle_S$ and $\M f$, $\M_{\mathcal D} f$ instead.

We record below some further basic results about the $\| \cdot \|_{L\log^{\sigma} L(X)}$-norm in the form of a proposition.

\begin{proposition}\label{prp:normequiv}
Let $(X, \mathcal{A}, \mu)$ be a finite measure space. For a measurable function $f:X\to\mathbb{C}$, the following hold.
\begin{enumerate}[itemsep=.5em]
\item[(i)] For $\sigma\geq 0$, the following are equivalent.
\begin{itemize}[itemsep=.5em]
\item[(a)] $ \|f\|_{L\log^{\sigma} L(X)}\le 1$.
\item[(b)] $ \int_X \Phi_{\sigma} (|f(x)|)\,\d\mu_X(x)\le 1 $.
\end{itemize}
\item[(ii)] For $\sigma \geq 0$, 
\[
\|f\|_{L\log^{\sigma} L(X)}
\le
1+\int_X \Phi_{\sigma} (|f(x)|)\,\d\mu_X(x).
\]
\item[(iii)] For $\sigma \geq 0$,
\[
\int_X \Phi_{\sigma} (|f(x)|)\,\d\mu_X(x)  \lesssim_{\sigma} \Phi_{\sigma} \left(\|f\|_{L\log^\sigma L(X)}\right),
\]
with implicit constant depending only on $\sigma$.
\end{enumerate}
\end{proposition}

The proof of Proposition \ref{prp:normequiv} follows from the definition of  the Luxemburg norm $\|\cdot \|_{L\log^{\sigma} L(X)}$ together with \eqref{eq:convex-scale} and \eqref{eq:submult}; we omit the details.

In this article we are interested in local endpoint estimates involving Orlicz spaces. We remark that it follows from Proposition \ref{prp:normequiv} that if $(X, \mathcal{A}, \mu)$ is a finite measure space and $\,T$ is a linear operator acting on measurable functions on $X$, then, for $\sigma \geq 0$ and any measurable function $f : X \to \mathbb{C}$, the statements
\[
\sup_{\alpha >0} \alpha \cdot \mu_X\left(\{x\in X:\, |Tf(x)|>\alpha\}\right)
\lesssim \|f\|_{L\log^\sigma L (X)}
\]
and
\[
\sup_{\alpha >0} \alpha\cdot  \mu_X\left(\{x\in X:\, |Tf(x)|>\alpha\}\right)
\lesssim  1+\int_X \Phi_\sigma(|f|)\,\d\mu_X 
\]
are equivalent. Furthermore, both endpoint estimates are implied by the modular estimate
\[
\mu_X\left(\{x\in X:\, |Tf(x)|>\alpha\}\right)
\lesssim  \int_X \Phi_\sigma\left(\frac{|f|}{\alpha}\right)\, \d \mu_X,\qquad \alpha>0.
\]
 
\subsection{Littlewood--Paley projections} We denote by $\{L_k\}_{k\in\Z}$ the collection of Littlewood--Paley intervals of the real line. Here, $L_k \coloneqq L_k^- \cup L_k^+$, where $L_k^- \coloneqq  [-2^{k+1}, - 2^k) $ and $L_k^+ \coloneqq (2^k, 2^{k+1}]$ are the dyadic frequency intervals at scale $2^k$.  

If $I$ is a bounded interval in $\mathbb{R}$ we denote by $P_I$ the rough Littlewood--Paley projection onto $I$. If $I \subset \mathbb{R}$ is a bounded interval, by a smoothed-out enlargement $S_I$ of $P_I$ we mean a multiplier operator whose symbol is given by $\psi( |I|^{-1} (\xi-c_I) )$, where $\psi $ is a $C^{\infty}$-function with $\ind_{[-1,1]} \le \psi \le \ind_{[-2,2]}$ and $c_I$ denotes the center of $I$. Note that $P_I S_I = S_I P_I = P_I$. 

Similarly, for $k \in \Z$, one defines $P_{L_k} \coloneqq P_{L_k^-} + P_{L_k^+}$ and we denote by $S_k$ the smoothed-out enlargement of $P_{L_k}$, defined in a way analogous to the (single interval) case presented above. In certain parts of the paper, given a smoothed-out enlargement $S_k$ of $P_{L_k}$, $\widetilde{S}_k$ will denote another smoothed-out enlargement of $P_{L_k}$ such that $\widetilde{S}_k S_k = S_k \widetilde{S}_k = S_k$.

\subsection{Additional remarks on notation throughout the paper} If $g$ is an $L^1$-function and $t > 0$, we denote by $g_t (x) \coloneqq t^{-1} g( t^{-1} x) $, $x \in \R$, the corresponding $L^1$-rescaling of $g$ at scale $t$. In several parts of the paper $L^1$-rescalings of the function $\phi (x) \coloneqq (1+|x|^2)^{-3/4}$, $x \in \R$, naturally arise. In certain parts of the paper we shall be able to use $L^1$-rescalings of the function $\omega (x) \coloneqq (1+|x|^2)^{-10}$, $x \in \R$, that decays to zero faster than $\phi$ for large $|x|$.

If $f \in L^1 (\R^n)$ and $S$ is a measurable subset of $\R^n$ of positive measure, we denote the average of $f$ over $S$ by $\langle f \rangle_S \coloneqq {|S|^{-1}} \int_S f(y) \, \d y$ (cf. \S \ref{sec:Orlicz}). 

If $f, g \in L^1(\R^n)$ we denote by $f \ast g$ the convolution of $f$ with $g$ and by $f \ast^{x_j} g$ the convolution in the $j$th variable, $j \in \{1, \ldots, n\}$. 

For a one-dimensional H\"ormander--Mihlin multiplier $m\colon\R\to\mathbb C$, we define its H\"ormander--Mihlin norm by
\[
\|m\|_{\mathrm{HM}(\R)} \coloneqq \sup_{j=0,1}\, \sup_{\xi\neq 0} |\xi|^{j}\left|\partial^j m(\xi)\right|.
\]

For $n=1$, if $m \colon \mathbb{R} \to \mathbb{C}$ is a Marcinkiewicz multiplier, its Marcinkiewicz norm is given by 
\[
\| m \|_{M (\mathbb{R})} \coloneqq \sup_{I \in \mathcal{L}} \| m \|_{BV (I)} + \| m \|_{L^{\infty} (\mathbb{R})},
\] 
where $\| m \|_{BV (I)}$ denotes the total variation of $m$ on $I \in \mathcal{L}$.
For $n=2$, if $m : \R^2 \to \mathbb{C}$ is a two-parameter Marcinkiewicz multiplier, we define its (two-parameter) Marcinkiewicz norm by
\[
\| m \|_{M (\mathbb{R}^{\otimes 2})} \coloneqq N_2 (m) + \| m \|_{L^{\infty} (\mathbb{R}^2)},
\] 
where
\[
N_2 (m) \coloneqq \sup_{\eta \in \R} \sup_{I \in \mathcal{L}} \int_I | \partial_1 m (\xi, \eta) | \, \d \xi + \sup_{\xi \in \R} \sup_{J \in \mathcal{L}} \int_J | \partial_2 m (\xi, \eta) | \, \d \eta + \sup_{I,J \in \mathcal{L}} \int_{I \times J} |\partial_1 \partial_2 m (\xi, \eta) | \, \d \xi \d \eta . 
\]
For $n \in \N$ with $n>2$, the ($n$-parameter) Marcinkiewicz norm $\| m \|_{M (\mathbb{R}^{\otimes n})}$ of an $n$-parameter Marcinkiewicz multiplier $m$ on $\R^n$ is defined in an analogous way.

Finally, throughout the paper, the implicit constants in $\lesssim$ may depend on $n$; we suppress this dependence without further comment.

\section{Global vector-valued estimates in the one-dimensional case}\label{sec:1-d_vv}
In this section we show how, by combining suitably adapted arguments from \cites{BCPV,TW}, one obtains appropriate $\ell^2$-valued endpoint estimates for $\mathcal{R}_2$--multipliers that will be used in the proof of our main results; see Proposition~\ref{prp:vvaluedTW} and Theorem~\ref{thm:vvaluedR_2} below.

\subsection{A weak-type estimate and an implicit square function characterization}\label{Recall} One of the main results in \cite{TW} is \cite{TW}*{Proposition 5.1}, which says the following weak-type estimate for frequency localized modulated singular integrals.

\begin{proposition}[Tao--Wright, \cite{TW}]\label{prp:TW} Let $\{I\}_{I \in \mathcal{I}}$ be a collection of frequency intervals with overlap at most $N$ and suppose that for each $I \in \mathcal{I}$ we are given a function $f_I$ and a nonnegative function $F_I$ such that
\[
|f_I|\lesssim \phi_{|I|^{-1}} \ast F_I.
\]
Then
\begin{equation}\label{eq:weakL1TW}
 \left|\left\{x\in\R:\, \left|\sum_{I \in \mathcal{I}} T_{m_I} f_I(x)\right|>\alpha \right\}\right|\lesssim\frac{N^{1/2}}{\alpha}\int_{\R} \left(\sum_{I \in \mathcal{I}} |F_I|^2\right)^{\frac{1}{2}},
\end{equation}
where for each $I \in \mathcal{I}$, $m_I$ is supported on $I$ and there exists $\xi_I\in I$ such that $m_I(\cdot+\xi_I)$ is a H\"ormander--Mihlin multiplier. The implicit constant is absolute.
\end{proposition}

In order to efficiently apply this proposition, the authors of \cite{TW} rely on an implicit square function characterization of $L\log^{1/2}L( [0,1] )$; see \cite{TW}*{Proposition 4.1}. We have the following generalization of \cite{TW}*{Proposition 4.1} from \cite{BCPV}.

\begin{proposition}[\cite{BCPV}*{Corollary E}] \label{prp:weakSF} Let $\sigma$ be a nonnegative integer. Let $J\subset \R$ be a bounded interval and let $f \in L\log^{\frac{\sigma+1}{2}}L(J)$. 

For $2^k\geq |J|^{-1}$ there exist functions $f_k$ with $\supp f_k\subset 4J$, satisfying $|S_k f|\lesssim \phi_{2^{-k}} \ast f_k$, and
\[
\left\|\left( \sum_{2^k\geq|J|^{-1}}|f_k|^2\right)^{\frac{1}{2}}\right\|_{L\log^{\frac{\sigma}{2}}L(J)} \lesssim \|f\|_{ L\log^{\frac{\sigma+1}{2} } L(J)}.
\]

If in addition $\int f=0$ then for $2^k<|J|^{-1}$ we can find functions $f_k$ satisfying $|S_k f|\lesssim \phi_{2^{-k}}\ast f_k$ and such that for every $\alpha>0$ the following modular inequality holds true
\[
\int_{\R} \Phi_{\frac{\sigma}{2}}\left(\alpha ^{-1}\left( \sum_{2^k<|J|^{-1}}|f_k|^2\right)^{\frac{1}{2}}\right) \lesssim \frac{1}{|J|}  \int_J \Phi_{\frac{\sigma}{2}}\left(\frac{|f|}{\alpha}\right).
\]
In this case, i.e. when $2^k<|J|^{-1}$, we do not claim anything about the support of the functions $f_k$.
\end{proposition}

\begin{remark}Specializing to the case $\sigma=0$ the estimates of Proposition~\ref{prp:weakSF} imply
\[
\frac{1}{|J|}\int_{\R}\left(\sum_{k\in\Z}|f_k|^2\right)^{\frac{1}{2}} \lesssim \|f\|_{L\log^{\frac{1}{2}}L(J)}.
\]
\end{remark}

The proof of Proposition~\ref{prp:weakSF} is essentially contained in \cite{BCPV}; we include it below, adapted to the present notation, for the reader's convenience.

\begin{proof}[Proof of Proposition~\ref{prp:weakSF}]The proof follows by the corresponding estimate for functions $f$ with $\supp f\subset [0,1]$. Indeed, for such functions we can use the Chang--Wilson--Wolff inequality \cite{CWW} and duality to construct, for $2^k\geq 1$, functions $f_k$ with the desired property; see \cite{BCPV}*{p.~1281--1284} for the details of this argument. By splitting any interval $[-M,M]$, with $M>0$ being an absolute constant, into at most $O(M)$ intervals of length $1$ and translating, the same is true for functions with $\supp f \subset [-M,M]$. Now replace $J$ by an interval $J^{\ast}$ with the same center as $J$ and dyadic length so that $|J|\leq |J^{\ast}| \leq C|J|$ for an absolute constant $1<C\leq 2$. We argue for $\supp f \subset J^{\ast}$.

Let $J^{\ast}=[x_0,x_0+|J^\ast|]$ be any interval with $L \coloneqq |J^{\ast}|$ dyadic. Define
\[
\widetilde{f} (u)\coloneqq f(x_0+Lu),\qquad u\in\mathbb{R},
\]
and note that $ \supp \widetilde f\subset [0,1]$. 
Since we assume $2^k\geq |J|^{-1}\geq L^{-1}$ we will have that if $k' \in \mathbb{Z}$ is such that $2^{k'} =  2^k L $ then $2^{k'} \geq 1$. For $x=x_0+Lu$ we have
\[
 (\phi_{2^{-k}} \ast f)(x_0+Lu) =(\phi_{2^{-k'}}\ast \widetilde f)(u).
 \]
By the normalized version of the proposition applied to $\widetilde f$ at index $k'$ with $2^{k'}\geq 1$, there exists
a function $\widetilde f_{k'}$ supported on $[-\tfrac{1}{3},\tfrac{4}{3}]$ such that
\[
 |  \phi_{2^{-k'}} \ast \widetilde{f}    |\ \lesssim\ \psi_{2^{-k'}} \ast \widetilde f_{k'}.
\]
Now define
\[
f_k(y)\coloneqq \widetilde {f}_{k'} \!\left(\frac{y-x_0}{L}\right).
\]
Then
\[
\supp f_k \subset x_0+L\cdot\left[-\tfrac13,\tfrac43\right]
=\left[x_0-\tfrac{L}{3},\,x_0+\tfrac{4L}{3}\right] \eqqcolon J_1 .
\]
Moreover, undoing the change of variables,
\[
(\psi_{2^{-k'}}\ast\widetilde f_{k'})(u)
=(\psi_{2^{-k}}\ast f_k)(x_0+Lu).
\]
Thus, for $x=x_0+Lu$,
\[
|(\phi_{2^{-k}}\ast f)(x)|
=|(\phi_{2^{-k'}}\ast\widetilde f)(u)| 
\lesssim (\psi_{2^{-k'}}\ast\widetilde f_{k'})(u)
=(\psi_{2^{-k}}\ast f_k)(x).
\]
This proves the rescaled estimate
\[
|\phi_{2^{-k}}\ast f|\ \lesssim\ \psi_{2^{-k}}\ast f_k
\]
with $\supp f_k\subset J_1\subset 4J$ as claimed. 

For $2^k<|J|^{-1}$ it is enough to observe that, for an appropriate enlargement $\widetilde{S}_k$ of $S_k$, by the mean zero condition on $f$ we have
\[
|\widetilde S_k f|\lesssim \omega_{2^{-k}} \ast | f | \lesssim \omega_{2^{-k}} \ast \left(2^k |J| \langle |f|\rangle_J \ind_J \right) .
\]
We set $f_k\coloneqq \widetilde S_k f$ for $2^k<|J|^{-1}$ and define
\[
H(x)\coloneqq
\left(\sum_{2^k<|J|^{-1}} \left(2^k |J| (\ind_{J} \ast \omega_{2^{-k}})(x)\right)^2\right)^{1/2}.
\]
Then
\begin{equation}\label{eq:Gfactor}
G \coloneqq \left(\sum_{2^k<|J|^{-1}} |f_k|^2\right)^{\frac{1}{2}} \lesssim \langle |f|\rangle_J\, H .
\end{equation} 
Note that $H(x)\lesssim 1$ for all $x\in \R$ and also $\|G\|_{L^1(\R)}\leq \langle |f|\rangle_J \|H\|_{L^1(\R)}\lesssim \langle |f|\rangle_J$. Letting $C_0\ge 1$ be such that $H(x)\le C_0$ for all $x\in \R$, we have for every $\alpha>0$ and all $x\in \R$,
\[
\Phi_{\frac{\sigma}{2}} \left(\frac{G(x)}{\alpha}\right)
=\Phi_{\frac{\sigma}{2}}\left(\frac{\langle |f|\rangle_J C_0 H(x)}{C_0\alpha}\right)
\leq C_0^{-1}H(x) \Phi_{\frac{\sigma}{2}} \left (\frac{C_0 \langle |f|\rangle_J}{\alpha}\right),
\]
where in the last inequality we used the convexity estimate \eqref{eq:convex-scale}. Integrating and using  again the fact that $\|H\|_{L^1(\R)}\lesssim 1$, we get for all $\alpha>0$
\begin{equation}\label{eq:prejensen}
\int_{\mathbb{R}} \Phi_{\frac{\sigma}{2}} \left(\frac{G(x)}{\alpha}\right) \, \d x
\le C_0 ^{-1}
\Phi_{\frac{\sigma}{2}} \left(\frac{C_0 \langle |f|\rangle_{J}}{\alpha}\right)
\int_{\mathbb{R}} H(x)\,\d x
\lesssim \Phi_{\frac{\sigma}{2}} \left(\frac{C_0\,\langle |f|\rangle_{J}}{\alpha}\right).
\end{equation}
Since $\Phi_{\frac{\sigma}{2}}$ is convex  we can use Jensen's inequality to estimate
\[
\Phi_{\frac{\sigma}{2}} \left(\frac{C_0\,\langle |f|\rangle_J}{\alpha}\right)
=
\Phi_{\frac{\sigma}{2}} \left( \frac{1}{|J|}  \int_J \frac{C_0 |f(y)|}{\alpha}\,\d y\right)
\le  \frac{1}{|J|} \int_J \Phi_{\frac{\sigma}{2}} \left (\frac{C_0 |f(y)|}{\alpha}\right)\,\d y.
\]
Inserting the estimate of the last display into \eqref{eq:prejensen}, we obtain
\[
\int_{\mathbb{R}} \Phi_{\frac{\sigma}{2}} \left(\frac{G(x)}{\alpha}\right)\,\d x
\lesssim \frac{1}{|J|}  \int_J \Phi_{\frac{\sigma}{2}} \left(\frac{C_0 |f(y)|}{\alpha}\right)\,\d y \lesssim \frac{1}{|J|}  \int_J \Phi_{\frac{\sigma}{2}} \left(\frac{|f(y)|}{\alpha}\right) \,\d y 
\]
by the submultiplicativity of $\Phi_{\frac{\sigma}{2}}$.
\end{proof}

\subsection{A global modular estimate in one parameter}\label{sec:gl_R} In this section we deal with a class of multipliers on $\R$ which we denote by $\mathcal R_2$. The precise definition of this class is given in Definition~\ref{def:R2n_atoms_norm} for the special case $n=1$. For the purposes of this section, and using the approximation argument from \cite{TW}*{p.~533--534}, we note that in order to prove a bound for an $\mathcal R_2$--multiplier on $\R$ it suffices to consider linear operators of the form
\[
   f\longmapsto  Tf=\sum_{I\in\mathcal I} \lambda_I P_I f,
\]
where the sum is over a collection of intervals $\mathcal I$ such that each $I$ is contained in a unique Littlewood--Paley interval $L_{k_I}$, $\sum_{I\in\mathcal I}\ind_I\leq N$ for some $N\in\N$, and the coefficients $\{\lambda_I\}_{I\in\mathcal I}$ satisfy
\[
   \sum_{\substack{I\in\mathcal I\\k_I=k}}|\lambda_I|^2 \leq \frac{1}{N}.
\]
Then, proving bounds for an arbitrary $\mathcal R_2$--multiplier reduces to bounding $T$ between the same spaces, uniformly in $N$.

The estimate of Proposition~\ref{prp:TW} may be combined with Proposition~\ref{prp:weakSF} and a suitable Calder\'on--Zygmund decomposition at the Orlicz scale to recover the global estimate for $\mathcal R_2$--multipliers, namely \eqref{eq:globaloneparam} below. This is a special case of \cite{BCPV}*{Theorem A}, where the estimate was proved via a slightly different combination of these ingredients. Thus, while this subsection contains no new result, we take the chance to explain this new argument here, since the same structure will be exploited below to prove the vector-valued analog of \eqref{eq:globaloneparam}.

Consider a function $f\in L\log^{\frac{1}{2}}L(\R)$ and perform a Calder\'on--Zygmund decomposition of $f$, similar to the corresponding one in the proof of \cite{BCPV}*{Proposition 5.2}. More precisely, consider the maximal dyadic intervals $J$ such that $\langle |f|\rangle_{\Phi_{1/2},J}>\alpha$ and define 
\[
f=g+b, 
\]
where
\[
g \coloneqq \left[f\ind_{ \R\setminus \bigcup_J J }  +\sum_J \langle f \rangle_{J}\ind_J\right] \quad \text{and} \quad b \coloneqq \sum_J \left(f-\langle f \rangle_{J}\right)\ind_J.
\]
We have, as usual, 
\[
\left|\bigcup_J J\right|\lesssim \int_{\R} \Phi_{1/2} \left(\frac{|f|}{\alpha}\right) ,\qquad \|g\|_{L^\infty(\R)}\lesssim \alpha,\qquad \langle |b_J|\rangle_{\Phi_{\frac{1}{2}},J}\lesssim \langle |f|\rangle_{\Phi_{1/2},J} \lesssim \alpha.
\]
Let $Tf=\sum_I \lambda_I P_I f$ be an $\mathcal R_2$--multiplier as described above. The good part $Tg$ is estimated directly in $L^2(\R)$ without invoking Proposition~\ref{prp:TW},
\[
 \left|\left\{x\in\R:\, |T(g)(x)|>\alpha \right\}\right|\lesssim\frac{1}{\alpha^2}\int_{\R}|g|^2\leq\frac{1}{\alpha}\int_{\R} |f|.
\]
We now move to the study of the bad part. Each atom has mean zero and $b_J\in L\log^{\frac{1}{2}}L(J)$. Proposition~\ref{prp:weakSF} applies and we get that for each bad interval $J$ and for each $k\in\Z$ there exists a function $F_{k,J}$ such that 
\[
 |S_k b_J|\lesssim \phi_{2^{-k}}\ast F_{k,J}.
\]
Note that $P_I = P_I S_I S_{k_I} $ for any smoothed-out enlargement $S_I$ of $P_I$ and any smooth Littlewood--Paley projection $S_{k_I}$ associated with the Littlewood--Paley interval $L_{k_I} \supset I$. We now define 
\[
f_{I,J} \coloneqq \lambda_I S_I S_{k_I}(b_J)
\]
and note that
\[
|f_{I,J}| \lesssim \phi_{ 2^{-k_I}} \ast \omega_{|I|^{-1}}\ast(\lambda_I F_{k_I,J})\lesssim \widetilde \omega_{|I|^{-1}} \ast (\lambda_I F_{k_I,J}).
\]
With this definition we have that
\[
Tb_J=\sum_I \lambda_I P_I b_J =\sum_I P_I (f_{I,J}).
\]
Moreover, by Proposition~\ref{prp:weakSF} for $\sigma = 0$, we have
\[
\int_\R \left(\sum_I |\lambda_I F_{k_I,J}|^2\right)^{\frac{1}{2}}\lesssim N^{-\frac{1}{2}}|J| \|b_J\|_{L\log^{\frac{1}{2}}L(J)}.
\]
Applying Proposition~\ref{prp:TW} to $b_J$ for each bad interval $J$, we get
\[
\left|\left\{|Tb_J|>\alpha\right\}\right|\lesssim N^{\frac{1}{2}}\alpha^{-1}\sum_J |J| N^{-\frac{1}{2}}\|b_J\|_{L\log^{\frac{1}{2}}L(J)}\lesssim \sum_J |J|\lesssim \int_{\R} \Phi_{1/2}\left(\frac{|f|}{\alpha}\right).
\]
Combining the good and bad parts yields the global modular estimate from \cite{BCPV}*{Theorem A}
\begin{equation}\label{eq:globaloneparam}
\left|\left\{x\in\R:\, |Tf(x)|>\alpha\right\}\right|\lesssim \int_{\R} \Phi_{1/2}\left(\frac{|f(x)|}{\alpha}\right)\, \d x.
\end{equation}

Of course this also implies the local estimate on any interval $K$
\[
\frac{1}{|K|}\left|\left\{x\in K:\, |T(f\ind_K)|>\alpha \right\}\right|\lesssim\frac{1}{\alpha} \|f\|_{L\log^{\frac{1}{2}}L(K)}.
\]
Remember that we use normalized (averaged) Orlicz norms.

\subsection{A vector-valued extension in the one-dimensional case} Let us now see how to get a vector-valued analog of Proposition~\ref{prp:TW}. The setup is as follows. For each $\tau \in \Z$ we let $\mathcal I_\tau$ be a family of intervals such that in each family the pointwise overlap is at most $N_\tau$, namely,
\[
 \sum_{I\in\mathcal I_\tau} \ind_I \leq N_\tau.
\]
For a family of functions $\vec f\coloneqq \{f_{\tau,I}\}_{\tau,I}$, we consider the operator $\mathcal{T}$ given by
\[
\mathcal{T} (\vec f)\coloneqq \left\{ \sum_{I\in\mathcal I_\tau} T_{m_{\tau,I}}f_{\tau,I} \right\}_\tau .
\]
The following proposition will play an important role in the proof of our main result, Theorem~\ref{thm:multiparam}.

\begin{proposition}\label{prp:vvaluedTW} Let $\{\mathcal I_\tau\}_\tau$ denote a collection of families of intervals in $\R$ such that
\[
\sum_{I\in\mathcal I_\tau}\ind_I \leq N_\tau\qquad \text{for each }\tau.
\] 
For each $\tau$ and $I\in\mathcal I_\tau$ let $m_{\tau,I}$ be supported in $I$ and suppose there exists $\xi_{\tau,I}\in I$ such that $m_{\tau,I}(\cdot+\xi_{\tau,I})$ is a H\"ormander--Mihlin multiplier and moreover, $\sup_{\tau,I}\|m_{\tau,I}(\cdot +\xi_{\tau,I})\|_{\mathrm{HM}(\R)}\leq 1$. Assume that for each $\tau$ and $I\in\mathcal I_\tau$ the function $\vec f\coloneqq \{f_{\tau,I}\}_{\tau,I}$ satisfies 
\[
|f_{\tau,I}|\lesssim \phi_{|I|^{-1}}\ast F_{\tau,I}.
\]
The following global estimate holds
\[ 
\left\| \mathcal T (\vec f)\right\|_{L^{1,\infty}(\ell^2 _\tau)} \lesssim \int_{\R} \left(\sum_\tau N_\tau \sum_{I\in\mathcal I_\tau }|F_{\tau,I}|^2\right)^{\frac{1}{2}}.
\]
\end{proposition}

The proof adapts the arguments in \cite{TW}*{Proposition 5.1} to the Hilbert-space-valued setting; we sketch the necessary modifications below.

\begin{proof}[Proof of Proposition~\ref{prp:vvaluedTW}] As in the proof of \cite{TW}*{Proposition~5.1}, we write
\[
\vec f\coloneqq \{f_{\tau,I}\}_{\tau,I},\qquad f_{\tau,I} \coloneqq a_{\tau,I}\,(F_{\tau,I} \ast \phi_{|I|^{-1}}),
\qquad \|a_{\tau,I}\|_{L^\infty}\lesssim 1. 
\]
Thus it suffices to control the operator
\[
\widetilde{\mathcal{T}}(\vec f)(x)
\coloneqq \left\{ \sum_{I\in\mathcal I_\tau}
T_{m_{\tau,I}}\left(a_{\tau,I}(F_{\tau,I}\ast\phi_{|I|^{-1}})\right)(x)
\right\}_{\tau} .
\]
To this end, introduce the scalar square function
\[
\mathsf F(x)\coloneqq \left( \sum_{\tau}\sum_{I\in\mathcal I_\tau}
N_\tau\,|F_{\tau,I}(x)|^2 \right)^{1/2}.
\]
The desired estimate is
\[
\left\| \widetilde{\mathcal{T}}(\vec f) \right\|_{L^{1,\infty}(\ell^2_\tau)} \lesssim \|\mathsf F\|_{L^1}.
\] 

We now fix a level $\alpha>0$ and perform the Calder\'on--Zygmund decomposition of the scalar function
$\mathsf F$:
\[
\mathsf F = g + \sum_{J} b_J,
\]
where $\{J\}_J$ is the family of maximal dyadic intervals such that $\langle \mathsf F \rangle_J >\alpha$. Note that if
\[
S\coloneqq \bigcup_J J,\qquad G\coloneqq \R\setminus S,
\]
then
\[
|S|=\sum_J |J| \lesssim \frac{1}{\alpha}\|\mathsf F\|_{L^1(\R)}.
\]
For each $(\tau,I)$ we write the Calder\'on--Zygmund decomposition
\[
F_{\tau,I} = g_{\tau,I} + \sum_J b_{J;\tau,I},
\]
where 
\[
g_{\tau,I}\coloneqq F_{\tau,I} \ind_G+\sum_J \langle F_{\tau,I}\rangle_J\ind_J  ,\qquad b_{J;\tau,I}\coloneqq (F_{\tau,I}-\langle F_{\tau,I}\rangle_J)\,\ind_J .
\]
By definition of the bad intervals we have
\[
\mathsf F(x) \ind_G (x) \leq \alpha.
\]
Therefore, we have the pointwise estimate
\begin{align*}
\left(\sum_{\tau,I} N_\tau |g_{\tau,I}(x)|^2\right)^{\frac{1}{2}} & \leq \left(\sum_{\tau,I} N_\tau |F_{\tau,I}(x)|^2\right)^{\frac{1}{2}} \ind_G+\sum_J\left(\sum_{\tau,I} N_\tau \langle |F_{\tau,I}|\rangle_J ^2\right)^{1/2}\ind_J 
\\
& \leq \mathsf F\ind_G + \sum_J \left\langle \left(\sum_{\tau,I}N_\tau |F_{\tau,I}|^2\right)^{\frac{1}{2}} \right\rangle_J\ind_J \leq \mathsf F\ind_G+\sup_J \langle \mathsf F\rangle_J 
\end{align*} 
and so,
\begin{equation}\label{eq:calc}
 \left(\sum_{\tau,I} N_\tau |g_{\tau,I}(x)|^2\right)^{\frac{1}{2}} 
\lesssim \alpha.
\end{equation}
Now let
\[
\vec g \coloneqq \{g_{\tau,I}\}.
\]
For the good part, $\vec g$, we use the $N_\tau$-overlap of each family $\mathcal I_\tau$ and the $L^2(\R)$-boundedness of the operator $h\mapsto a_{\tau,I}(h\ast\phi_{|I|^{-1}})$ and of each $T_{m_{\tau,I}}$, to write
\begin{align*} 
\left|\left\{x\in\R:\, \left\| \widetilde{\mathcal{T}} (\vec g) \right\|_{\ell^2 _\tau}>\alpha \right\}\right|  &\leq\frac{1}{\alpha^2}\int_\R \left\| \widetilde{\mathcal{T}} (\vec g) \right\|_{\ell^2 _\tau} ^2  
 =\sum_{\tau} \left\| \sum_{I\in\mathcal I_\tau} T_{m_{\tau,I}}\left(a_{\tau,I}(g_{\tau,I}\ast\phi_{|I|^{-1}})\right)
\right\|_{L^2(\R)}^2  
\\
& \lesssim \sum_\tau N_\tau \sum_{I\in\mathcal I_\tau}\|g_{\tau,I}\|_{L^2(\R)}^2 
 \lesssim \alpha\|\mathsf F\|_{L^1(\R)},
\end{align*}
where we also used the calculation in \eqref{eq:calc} in passing to the last estimate.
	
We move to the analysis of the bad part by estimating the contribution of
\[
B\coloneqq \left(\sum_{\tau} \left| \sum_J\sum_{I\in\mathcal I_\tau} T_{m_{\tau,I}}\left(a_{\tau,I} \left(b_{J;\tau,I}\ast\phi_{|I|^{-1}} \right)\right) \right|^2\right)^{1/2}
\]
to the distribution function. We will prove the uniform estimate
\[
|\{x \in \R: \, B(x)>\alpha\}|
 \lesssim \alpha^{-1}\|\mathsf F\|_{L^1(\R)}.
\]
We split $B$ according to scales $|I||J|>1$ or $|I||J|\leq 1$, namely $B \le B^{\rm{sm}}+ B^{\rm{lrg}}$, where
\[
B ^{\rm{sm}}\coloneqq
\left(\sum_{\tau}
\left|\sum_J
\sum_{\substack{I\in\mathcal I_\tau\\ |I||J|\le 1}}
T_{m_{\tau,I}}\left(a_{\tau,I}\left(b_{J;\tau,I}\ast\phi_{|I|^{-1}}\right)\right)\right|^2\right)^{1/2}
\]
and
\[
B ^{\rm{lrg}}\coloneqq
\left(\sum_{\tau}
\left|\sum_J
\sum_{\substack{I\in\mathcal I_\tau\\ |I||J| > 1}}
T_{m_{\tau,I}}\left(a_{\tau,I}\left(b_{J;\tau,I}\ast\phi_{|I|^{-1}}\right)\right)\right|^2\right)^{1/2}.
\]
To handle $B^{\rm{sm}}$, note that by Chebyshev,
\[
\begin{split}
|\{B ^{\rm sm}>\alpha\}|
&\leq \alpha^{-2}
\sum_{\tau}
\left\| 
\sum_{\substack{I\in\mathcal I_\tau }}
T_{m_{\tau,I}}\left(\sum_{\substack{J:\, |I||J|\le 1}}a_{\tau,I}\left(b_{J;\tau,I}\ast\phi_{|I|^{-1}}\right)\right)
\right\|_{2}^{2} \\
&\lesssim\alpha^{-2}
\sum_{\tau} N_\tau
\sum_{\substack{I\in\mathcal I_\tau}}
\left\|\sum_{J:\, |I||J|\le 1}b_{J;\tau,I}\ast\phi_{|I|^{-1}}\right\|_2^2.
\end{split}
\]
Using the cancellation of each $b_{J;\tau,I}$ we obtain the estimate
\[
\left|\left(b_{J;\tau,I} \ast \phi_{|I|^{-1}}\right)(x)\right|\lesssim \left(\phi_{|I|^{-1}} \ast \left(|I||J|\langle |b_{J;\tau,I}|\rangle_J\ind_J\right)\right)(x).
\]
Using the last two displays and the Fefferman--Stein inequality \cite{CFef_St}, together with the fact that the stopping condition implies $\langle |b_{J;\tau,I}|\rangle_J \lesssim \langle |F_{\tau,I}|\rangle_J$, we get
\[
\begin{split}
|\{B ^{\rm sm}>\alpha\}| &\lesssim \alpha^{-2}\sum_\tau N_\tau\sum_J \sum_{\substack{I\in\mathcal I_\tau\\ |I||J|\le 1}} \left(|I||J|\langle |b_{J;\tau,I}|\rangle_J\right)^2 |J|
  \leq \alpha^{-2}\sum_J |J|\sum_\tau N_\tau \sum_{\substack{I\in\mathcal I_\tau\\ |I||J|\le 1}} \langle |b_{J;\tau,I}|\rangle_J ^2
\\
& \leq \alpha^{-2}\sum_J |J|\sum_\tau  \sum_{I\in\mathcal I_\tau} N_\tau \langle |F_{\tau,I}|\rangle_J^2  \lesssim \alpha^{-2}\sum_J|J| \langle \mathsf F\rangle_J ^2  \leq \sum_J |J|
 \\
& \lesssim \alpha^{-1}\|\mathsf F\|_{L^1(\R)}.
\end{split}
\] 

We move to the study of the term $B^{\rm{lrg}}$ corresponding to the large frequency intervals, which is the main term. We immediately split this term into a local and nonlocal piece
\[
B ^{\rm lrg} \leq B^{\rm loc}+B^{\rm nloc},
\]
where
\[
B^{\rm loc} \coloneqq \left(\sum_{\tau}
\left|\sum_J
\sum_{\substack{I\in\mathcal I_\tau\\ |I||J|>1 }}
T_{m_{\tau,I}}\left(a_{\tau,I}\ind_{2J}\left(b_{J;\tau,I}\ast\phi_{|I|^{-1}}\right)\right)\right|
^2\right)^{1/2}
\]
and
\[
B^{\rm nloc} \coloneqq \left(\sum_{\tau} \left|\sum_J \sum_{\substack{I\in\mathcal I_\tau\\ |I||J|>1 }} T_{m_{\tau,I}}\left(a_{\tau,I}\left(1-\ind_{2J}\right) \left(b_{J;\tau,I}\ast\phi_{|I|^{-1}}\right)\right)\right| ^2\right)^{1/2}.
\]
The nonlocal term $B^{\rm nloc}$ can be easily dealt with because on $\R\setminus 2J$ we have the pointwise estimate
\[
\left|\sum_J a_{\tau,I}(1-\ind_{2J})(b_{J;\tau,I}\ast\phi_{|I|^{-1}})\right|\lesssim \phi_{|I|^{-1}} \ast \left( \sum_J \langle|b_{J;\tau,I}|\rangle_J \ind_J\right)\lesssim  \phi_{|I|^{-1}}\ast \left(\sum_J \langle|F_{\tau,I}|\rangle_J \ind_J\right),
\]
which does not require cancellation. Thus we may bound $|\{B^{\rm nloc}>\alpha\}|$ by passing to $L^2(\R)$ and using the Fefferman--Stein inequality as before.

So we focus now on the local term for the large frequency intervals, namely the term $B^{\rm loc}$: this is a suitable adaptation of the arguments in \cite{TW}.

Fix a Schwartz function $\psi$ with $\widehat\psi\equiv 1$ on $[-1/4,1/4]$ and
$\supp\widehat\psi\subset[-1,1 ]$, and, as usual, we set
\[
\psi_{|J|}(x)\coloneqq |J|^{-1}\psi(x/|J|).
\]
For each $I$ choose any $\xi_I\in I$ and define
\[
P_{J,I}\coloneqq \mathrm{Mod}_{\xi_I} (\psi_{|J|}\ast\cdot) \mathrm{Mod}_{-\xi_I},
\qquad
Q_{J,I}\coloneqq \mathrm{Id} - P_{J,I}.
\]
Here, for $\xi \in \R$, $\mathrm{Mod}_{\xi}$ denotes the modulation operator given by
$\mathrm{Mod}_{\xi} h(x)\coloneqq e^{2\pi i \xi x}h(x)$, $x \in \R$.

Note that $P_{J,I}$ is a smooth frequency cutoff to a window of width $\sim |J|^{-1}$
around the frequency singularity $\xi_I$. We split
\[
T_{m_{\tau,I}}
=
T_{m_{\tau,I}}P_{J,I}+T_{m_{\tau,I}}Q_{J,I},
\]
and, replacing $T_{m_{\tau,I}}$ in the definition of $B^{\rm loc}$ by $T_{m_{\tau,I}}P_{J,I}$ and $T_{m_{\tau,I}}Q_{J,I}$, we get $B^{\rm loc} _P$ and $B^{\rm loc} _Q$, respectively. We then have
\[
B ^{\rm loc } \le B^{\rm loc} _{P} +B^{\rm loc} _{Q}.
\]

We first treat the term $B^{\rm loc} _{P}$. Using symbol estimates we have
\[
\left| P_{J,I}\left(a_{\tau,I}\ind_{2J}\left(b_{J;\tau,I}\ast\phi_{|I|^{-1}}\right)\right)\right|\lesssim \M(\ind_J)^{10}\langle|b_{J;\tau,I}  |\rangle_J\lesssim  \M(\ind_J)^{10}\langle|F_{\tau,I}  |\rangle_J.
\]
Thus the estimate for the term $B^{\rm loc} _{P}$ follows by arguing as before.

Finally, let us look at the term $B^{\rm loc}_{Q}$. Repeating verbatim the proof in the corresponding part of \cite{TW} one gets
\[
\left|\left\{B_Q ^{\rm loc}>\alpha \right\} \right|\leq\left|\bigcup_J 2J \right|+ \alpha^{-1}\sum_J |J|^{-\frac{1}{2}} \left(\sum_\tau N_\tau \sum_{I:\, |I||J|>1} |J|^4 \langle |b_{J;\tau,I}|\rangle_J ^2 |I(J)|\right)^{\frac{1}{2}},
\]
where $I(J)\subset I$ is an interval with $|I(J)|\sim |J|^{-1}$. We can then close this estimate since 
\[
\begin{split}
\sum_J |J|^{-\frac{1}{2}} \left(\sum_\tau N_\tau \sum_{I:\, |I||J|>1} |J|^4 \langle |b_{J;\tau,I}|\rangle_J ^2 |I(J)|\right)^{\frac{1}{2}}   &\lesssim \sum_J |J|\left(\sum_{\tau,I}N_\tau \langle |F_{\tau,I}|\rangle_J ^2\right)^{\frac{1}{2}} 
 \lesssim \sum_J |J| \langle \mathsf F\rangle_J 
\\
& \lesssim \sum_J |J| \alpha   \lesssim \|\mathsf F\|_{L^1(\R)}.\qedhere
\end{split}
\]
\end{proof}

In the scalar case, presented in \S~\ref{sec:gl_R}, Proposition~\ref{prp:TW} was combined with the implicit square function characterization of $L\log^{\frac{\sigma}{2}}L$ spaces, Proposition~\ref{prp:weakSF}, together with a Calder\'on--Zygmund decomposition at the Orlicz scale, to yield the global modular endpoint estimate for $\mathcal R_2$--multipliers. Indeed, the implicit square function characterization of $L\log^{1/2}L$ requires cancellation in order to deal with the Littlewood--Paley intervals $[\pm 2^{k_I},\pm 2^{k_I+1}]$ when working with atoms supported on some interval $J$ with $|J|2^{k_I}<1$. However, we saw that this can be dealt with via the Calder\'on--Zygmund decomposition obtained in terms of Orlicz stopping conditions $\langle |f|\rangle_{\Phi_{1/2},J}>\alpha$. 

In the present vector-valued setup, Proposition~\ref{prp:TW} will be replaced by Proposition~\ref{prp:vvaluedTW}. Additionally, we will use a vector-valued Calder\'on--Zygmund decomposition relying on the Fefferman--Stein endpoint estimate for the Orlicz maximal operator $\M_{\Phi_{1/2}}$. More generally, the following version of the Fefferman--Stein inequality for Orlicz maximal operators holds. This is probably known, but we include the proof for completeness.

\begin{proposition}\label{prp:OrliczFS}
Let $\Phi_\sigma(t)=t[\log(e+t)]^\sigma$, $\sigma\geq 0$, and let $\M_{\Phi_\sigma}$ be the associated Orlicz maximal operator. Then
\[
\left|\left\{x\in\R^n:\, \left(\sum_j |\M_{\Phi_\sigma} f_j(x)|^2\right)^{1/2}>\alpha\right\}\right| \lesssim_{\sigma,n} \int_{\R^n}\Phi_\sigma\left(\frac{\left(\sum_j|f_j(x)|^2\right)^{1/2}}{\alpha}\right)\,\d x.
\]
\end{proposition}

\begin{proof} It suffices to prove the desired inequality for the dyadic maximal operator $\M_{\mathcal D,\Phi_\sigma}$. To that end, set $F\coloneqq (\sum_j|f_j|^2)^{1/2}$. Apply the Calder\'on--Zygmund decomposition associated with $\M_{\mathcal D,\Phi_\sigma}$ to $F$ at height $\alpha$. This produces a pairwise disjoint collection of maximal dyadic cubes $\mathcal Q=\{Q\}$ with 
\[
\Omega\coloneqq\{\M_{\mathcal D,\Phi_\sigma}F>\alpha\}=\bigcup_{Q\in\mathcal Q}Q,\qquad \langle F\rangle_{\Phi_\sigma,Q}\simeq_n\alpha \quad \forall Q\in\mathcal Q,\qquad |\Omega|\lesssim\int\Phi_\sigma\left(\frac F\alpha\right).
\]
Now we split $f_j=g_j+b_j$ with $g_j \coloneqq f_j\ind_{\Omega^c}$ and $b_j \coloneqq f_j\ind_\Omega$. Since $F\leq\alpha$ a.e.\ on $\Omega^c$, the $L^2$-estimate for the scalar $\M_{\Phi_\sigma}$ from \cite{CUP12}*{Proposition~A.1} gives
\[
\left\|\left(\sum_j|\M_{\mathcal D,\Phi_\sigma}g_j|^2\right)^{1/2}\right\|_2^2\lesssim_{\sigma,n}\int_{\Omega^c}F^2\leq\alpha\int_{\Omega^c}F,
\]
so by Chebyshev's inequality $|\{(\sum_j|\M_{\mathcal D,\Phi_\sigma}g_j|^2)^{1/2}>\alpha/2\}|\lesssim\int\Phi_\sigma(F/\alpha)$.

For the bad part, set $\widetilde f_j=\sum_{Q\in\mathcal Q}\langle|f_j|\rangle_{\Phi_\sigma,Q}\ind_Q$. By Minkowski's inequality for Orlicz spaces, we have 
\[
\left(\sum_j\left\langle|f_j|\right\rangle_{\Phi_\sigma,Q}^2\right)^{1/2}\lesssim_\sigma\langle F\rangle_{\Phi_\sigma,Q}\lesssim\alpha
\]
for every $Q\in\mathcal Q$; hence $(\sum_j|\widetilde f_j|^2)^{1/2}\lesssim\alpha\ind_\Omega$, and therefore $\int\sum_j|\widetilde f_j|^2\lesssim\alpha^2|\Omega|$.

Now let $\Omega^\ast\coloneqq \bigcup_{Q\in\mathcal Q}3Q$, so $|\Omega^\ast|\lesssim|\Omega|\lesssim\int\Phi_\sigma(F/\alpha)$. We claim that
\begin{equation}\label{eq:pointwise}
\M_{\mathcal D,\Phi_\sigma}b_j(x)\lesssim_{\sigma,n}\M_{\mathcal D,\Phi_\sigma}\widetilde f_j(x),\qquad x\notin\Omega^\ast.
\end{equation}
Indeed, fix such an $x$ and any dyadic $R\ni x$; if $R$ meets no $Q\in\mathcal Q$ there is nothing to prove, so suppose $R\cap Q\neq\varnothing$ for some $Q\in\mathcal Q$. Then $x\notin 3Q$ forces $Q\subset R$.

Set $\alpha_0\coloneqq\langle|f_j|\rangle_{\Phi_\sigma,Q}$. By definition of the Luxemburg norm, $\int_Q\Phi_\sigma(|f_j|/\alpha_0)\le|Q|$. Also, the submultiplicativity of $\Phi_\sigma$ gives $\Phi_\sigma(|f_j|/\rho)\lesssim_\sigma\Phi_\sigma(\alpha_0/\rho)\,\Phi_\sigma(|f_j|/\alpha_0)$ pointwise on $Q$ for any $\rho>0$. Integrating over $Q$, we get that for every $\rho>0$, there holds
\begin{equation}\label{eq:basic}
\int_Q\Phi_\sigma\left(\frac{|f_j|}{\rho}\right)\lesssim_\sigma\Phi_\sigma\left(\frac {\alpha_0}{\rho}\right)\int_Q\Phi_\sigma\left(\frac{|f_j|}{\alpha_0}\right)\le\Phi_\sigma\left(\frac{\alpha_0}{\rho}\right)|Q|,
\end{equation}
using the bound on $\alpha_0$ above. Now choose $\rho=\rho_0\coloneqq\langle \alpha_0\ind_Q\rangle_{\Phi_\sigma,R}$. By the definition of the Luxemburg norm, we have
\[
\frac{1}{|R|}\int_R\Phi_\sigma\left(\frac{\alpha_0\ind_Q(x)}{\rho_0}\right)\,\d x\le1.
\]
Also, since $\alpha_0\ind_Q$ vanishes outside $Q$, 
\[
1\geq \frac{1}{|R|}\int_R\Phi_\sigma\left(\frac{\alpha_0\ind_Q(x)}{\rho_0} \right)\,\d x=\frac{1}{|R|}\int_Q\Phi_\sigma\left(\frac{\alpha_0\ind_Q (x)}{\rho_0}\right)\,\d x=\frac{|Q|}{|R|}\,\Phi_\sigma\left(\frac{\alpha_0}{\rho_0}\right)\implies 
\Phi_\sigma\left(\frac{\alpha_0}{\rho_0}\right)\le\frac{|R|}{|Q|}.
\]
Plugging the inequality above into \eqref{eq:basic} yields
\[
\int_Q\Phi_\sigma\left(\frac{|f_j|}{\rho_0}\right)\lesssim_\sigma\Phi_\sigma\left(\frac{\alpha_0}{\rho_0}\right)|Q|\le |R|,
\]
so $\frac{1}{|R|}\int_Q\Phi_\sigma(|f_j|/\rho_0)\leq C_\sigma$ for some constant $C_\sigma>1$. Now \eqref{eq:convex-scale} implies that $\Phi_\sigma(t/C_\sigma)\le\Phi_\sigma(t)/C_\sigma$. Thus,
\[
\frac{1}{|R|}\int_Q\Phi_\sigma\left(\frac{|f_j|}{C_\sigma\rho_0}\right)\,\d x\le\frac{1}{C_\sigma}\frac{1}{|R|}\int_Q\Phi_\sigma\left(\frac{|f_j|}{\rho_0}\right)\,\d x\le1,
\]
which is exactly the statement $\langle |f_j|\ind_Q\rangle_{\Phi_\sigma,R}\le C_\sigma\rho_0=C_\sigma \alpha_0\langle\ind_Q\rangle_{\Phi_\sigma,R}$ by the homogeneity of the Luxemburg norm. We conclude
\[
\langle |f_j|\ind_Q\rangle_{\Phi_\sigma,R}\lesssim_\sigma \alpha_0\langle\ind_Q\rangle_{\Phi_\sigma,R}.
\]
Since $b_j=\sum_{Q\in\mathcal Q} f_j \ind_Q$, summing the estimate above over the disjoint cubes $Q\subset R$ yields 
\[
   \langle |b_j| \rangle_{\Phi_\sigma,R}\lesssim_\sigma\langle|\widetilde f_j|\rangle_{\Phi_\sigma,R}.
\]
Since $R\ni x$ was arbitrary, taking the supremum over $R$ gives the claimed pointwise bound \eqref{eq:pointwise}.

Now the scalar  $L^2-L^2$ estimate applied to $\widetilde f_j$ and Chebyshev's inequality immediately give
\[
\left|\left\{x\notin\Omega^\ast:\, \left( \sum_j|\M_{\mathcal D,\Phi_\sigma}b_j(x)|^2\right)^{1/2}>\frac\alpha2\right\}\right|\lesssim\frac{1}{\alpha^2}\int\sum_j|\widetilde f_j|^2\lesssim|\Omega|\lesssim\int\Phi_\sigma(F/\alpha).
\]
Combining the estimates for $g_j$ and $b_j$ proves the claim.
\end{proof}

The other main ingredient needed to combine with Proposition~\ref{prp:vvaluedTW} is the following vector-valued version of the implicit square function characterization of $L\log^{\sigma/2}L$, whose proof requires only a small adjustment of the argument in \cite{BCPV}*{Corollary E}.

\begin{proposition}\label{prp:weakSFvv} Let $\sigma$ be a nonnegative integer, $J\subset \R$ a bounded interval, and let $\{f_\tau\}_\tau$ be a sequence of functions supported in $J$ and such that
\[
 F\coloneqq \left(\sum_\tau |f_\tau|^2\right)^{\frac{1}{2}}\in L\log^{\frac{\sigma+1}{2}}L(J).
\]
If $2^k\geq |J|^{-1}$, there exist functions $\{f_{\tau,k}\}_{\tau,k}$ with $\supp f_{\tau,k}\subset 4J$ such that 
\[
 |S_k f_\tau|\lesssim \phi_{2^{-k}}\ast f_{\tau,k},\qquad \left\| \left(\sum_{\tau,\, 2^k\geq |J|^{-1}}|f_{\tau,k}|^2\right)^{\frac{1}{2}}\right\|_{L\log^{\frac{\sigma}{2}}L(J)}\lesssim \|F\|_{L\log^{\frac{\sigma+1}{2}}L(J)}.
\]
If in addition $\int f_\tau=0$ for all $\tau$, then, for $2^k< |J|^{-1}$, there exist functions $\{f_{\tau,k}\}_{\tau,k}$ satisfying $|S_k f_{\tau}|\lesssim \phi_{2^{-k}}\ast f_{\tau,k}$, and for every $\alpha>0$
\[
\int_{\R} \Phi_{\frac{\sigma}{2}}\left(\alpha ^{-1}\left(\sum_\tau \sum_{2^k<|J|^{-1}}|f_{\tau,k}|^2\right)^{\frac{1}{2}}\right)\lesssim \frac{1}{|J|} \int_J \Phi_{\frac{\sigma}{2}}\left(\frac{|F|}{\alpha}\right).
\]
In this case, i.e. for $2^k<|J|^{-1}$, we do not claim anything about the support of the functions $f_{\tau,k}$.
\end{proposition}

\begin{proof} We may assume that $J=[0,1]$. The proof relies on the vector-valued Chang--Wilson--Wolff inequality:
\[
\left\| \left(\sum_\tau |f_\tau-\mathbb E_0 f_\tau|^2\right)^{\frac{1}{2}}\right\|_{L^p([0,1])} \leq C p^{\frac{1}{2}} \left\|\left(\sum_{\tau,k}|\mathbb D_k f_{\tau}|^2\right)^{\frac{1}{2}} \right\|_{L^p([0,1])},\qquad p>2.
\]
This follows in turn from the corresponding $L^2(w)$ estimate, with constant $C[w]_{A_\infty}^{\frac{1}{2}}$; see \cite{WilsonBook}*{Theorem~3.4}.

The rest of the proof follows the duality argument from the proof of \cite{BCPV}*{Theorem D}. Indeed, one constructs functions $f_{\tau,k}$ such that $\mathbb D_k f_\tau = f_{\tau,k}$ and such that
\[
\left\| \left(\sum_{\tau,k}|f_{\tau,k}|^2\right)^{\frac{1}{2}}\right\|_{L\log^{\frac{\sigma}{2}}L([0,1])}\lesssim \left\| \left(\sum_\tau |f_\tau|^2\right)^{\frac{1}{2}}\right\|_{L\log^{\frac{\sigma+1}{2}}L([0,1])}. 
\]
This gives the result for $k\geq 0$, while for $k<0$ it is enough to set $f_{\tau,k}\coloneqq S_k f_\tau$. Under the assumption $\int f_\tau=0$ for all $\tau$, we can estimate $f_{\tau,k}$ pointwise. The proof of the claimed estimate in this case is exactly the same as in the scalar proof of Proposition~\ref{prp:weakSF}.
\end{proof}

Combining the implicit square function characterization of $L\log^{\sigma/2}L$, namely Proposition~\ref{prp:weakSFvv}, with Proposition~\ref{prp:vvaluedTW} yields the following vector-valued endpoint estimate for $\mathcal R_2$--multipliers. For the definition of the class of one-parameter $\mathcal{R}_2$--multipliers, see Section~\ref{sec:R2n}, and in particular Definition~\ref{def:R2n_atoms_norm} for $n=1$.

\begin{theorem}\label{thm:vvaluedR_2} Let $\{m_\tau\}_\tau$ be a family of $\mathcal R_2$--multipliers on the real line with uniform $\mathcal R_2$--norm $\sup_{\tau}\|m_\tau\|_{\mathcal R_2}\leq 1$. Then
\[
 \left|\left\{x\in \R:\, \left(\sum_\tau |T_{m_\tau}f_\tau(x)|^2\right)^{\frac{1}{2}}
>\alpha\right\}\right|\lesssim \int_{\R} \Phi_{\frac{1}{2}}\left(\frac{\|\{f_\tau(x)\}\|_{\ell^2 _\tau}}{\alpha}\right)\,\d x,\qquad \alpha>0.
\]
Also, for any bounded interval $K$ and nonnegative integer $\sigma$, if $\supp f_\tau \subset K$ for all $\tau$, then 
\[
 \left\| \left(\sum_\tau |T_{m_\tau}f_\tau|^2\right)^{\frac{1}{2}}\right\|_{L\log^{\frac{\sigma}{2}}L(K)}\lesssim \left\| \left(\sum_\tau |f_\tau|^2\right)^{\frac{1}{2}}\right\|_{L\log^{\frac{\sigma+3}{2}}L(K)}.
\]
\end{theorem}

\begin{proof} The proof follows the same coarse structure as the scalar proof in \S\ref{sec:gl_R}, so we only sketch the steps in the vector-valued case.

As in the one-parameter case, by the approximation argument of \cite{TW}*{p.~533--534} applied to each $m_\tau$, it suffices to treat the model operator 
\[
\mathcal T(\vec f)=\{T_\tau f_\tau\}_\tau,\qquad T_\tau f_\tau=\sum_{I\in\mathcal I_\tau}\lambda_{\tau,I}P_If_\tau, 
\]
 where, for some collection of positive integers $\{N_\tau\}_\tau$, there holds
 \[
    \sup_\tau \sum_{I\in \mathcal I_\tau}  \ind_I \leq N_\tau,\qquad \sup_{\tau}\sup_{k\in\Z} \sum_{\substack{I\in\mathcal I_\tau\\ k_I=k}}|\lambda_{\tau,I}|^2 \leq \frac{1}{N_\tau}.
 \]

Set $F\coloneqq (\sum_\tau|f_\tau|^2)^{1/2}$ and consider the maximal dyadic intervals $J$ with $\langle F\rangle_{\Phi_{1/2},J}>\alpha$. We set $f_\tau=g_\tau+b_\tau$ with $g_\tau\coloneqq f_\tau\ind_{\R\setminus\bigcup_JJ}+\sum_J\langle f_\tau\rangle_J\ind_J$ and $b_\tau\coloneqq\sum_J(f_\tau-\langle f_\tau\rangle_J)\ind_J=\sum_Jb_{\tau,J}$. By Proposition~\ref{prp:OrliczFS} we have $|\bigcup_JJ|\lesssim\int\Phi_{1/2}(F/\alpha)$, and by Minkowski's inequality $(\sum_\tau|g_\tau|^2)^{1/2}\lesssim\alpha$ pointwise, while $\langle(\sum_\tau|b_{\tau,J}|^2)^{1/2}\rangle_{\Phi_{1/2},J}\lesssim\langle F\rangle_{\Phi_{1/2},J}\lesssim\alpha$ for each $J$.

The good part is estimated directly in $L^2(\ell^2_\tau)$, using the $L^2(\ell^2_\tau)$-boundedness of the vector-valued operator $\mathcal T$:
\[
\left|\left\{\left(\sum_\tau|T_\tau g_\tau|^2\right)^{1/2}>\alpha\right\}\right|\lesssim\frac{1}{\alpha^2}\int_\R\sum_\tau|g_\tau|^2\lesssim\frac{1}{\alpha}\int_\R F.
\]

For the bad part, let us fix a bad interval $J$. Then, each $b_{\tau,J}$ has mean zero and also $(\sum_\tau|b_{\tau,J}|^2)^{1/2}\in L\log^{1/2}L(J)$. Because of the cancellation conditions on the $\{b_{\tau,J}\}_\tau$, the full strength of Proposition~\ref{prp:weakSFvv} applies to $\{b_{\tau,J}\}_\tau$ on $J$. Thus, for each $k\in\Z$ there exist functions $\{F_{\tau,k,J}\}_\tau$ with $|S_kb_{\tau,J}|\lesssim\phi_{2^{-k}}\ast F_{\tau,k,J}$. As in the scalar case, writing $P_I=P_IS_IS_{k_I}$ and setting $f_{\tau,I,J}\coloneqq\lambda_{\tau,I}S_IS_{k_I}(b_{\tau,J})$ gives $|f_{\tau,I,J}|\lesssim\widetilde\omega_{|I|^{-1}}\ast(\lambda_{\tau,I}F_{\tau,k_I,J})$ and $T_\tau b_{\tau,J}=\sum_IP_I(f_{\tau,I,J})$. By Proposition~\ref{prp:weakSFvv} with $\sigma=0$,
\[
\int_\R\left(\sum_\tau\sum_I|\lambda_{\tau,I}F_{\tau,k_I,J}|^2\right)^{1/2}\lesssim|J|\left\|\left(\sum_\tau|b_{\tau,J}|^2\right)^{1/2}\right\|_{L\log^{1/2}L(J)}.
\]
Applying Proposition~\ref{prp:vvaluedTW} to $\{b_{\tau,J}\}_\tau$ for each bad $J$ and summing over $J$ gives
\[
\left|\left\{\left(\sum_\tau|T_\tau b_\tau|^2\right)^{1/2}>\alpha\right\}\right|\lesssim\sum_J|J|\lesssim\int_\R\Phi_{1/2}(F/\alpha).
\]
Combining the good and bad part estimates, and using submultiplicativity of $\Phi_{1/2}$ we obtain
\[
\left|\left\{\left(\sum_\tau|T_\tau f_\tau|^2\right)^{1/2}>\alpha\right\}\right|\lesssim\int_\R\Phi_{1/2}\left(\frac{F}{\alpha}\right),
\]
which is the claimed weak-type estimate.

Now for each bounded interval $K$ the global estimate implies
\[
 \frac{1}{|K|}\left|\left\{x\in K:\, \left(\sum_\tau |T_{m_\tau}f_\tau(x)|^2\right)^{\frac{1}{2}}
>\alpha\right\}\right|\lesssim \frac{1}{\alpha} \left\|\left(\sum_\tau|f_\tau|^2\right)^{1/2}\right\|_{L\log^{1/2}L(K)}.
\]
The second conclusion in the statement of the Theorem follows by a standard level set integration argument, as for example in \cite{Bak2019}*{Lemma 3.2}.
\end{proof}

\section{Multiparameter $\mathcal{R}_2$--multipliers}\label{sec:R2n} In this section we introduce the class of multiparameter $\mathcal{R}_2$--multipliers. In order to do so, we need some additional notation. 

We work in $n$ parameters, $n \in \N$. For $j\in\{1,\ldots,n\}$ and $x=(x_1,\ldots,x_n)\in\R^n$ we denote by $\hat x_j \in \R^{n-1}$ the vector formed by omitting the $j$th coordinate of $x$.

\begin{definition}[$\mathcal R_{2,n}$--atoms and $\mathcal R_{2,n}$--norm]\label{def:R2n_atoms_norm}
Fix $n \in \N$. For each $j\in\{1,\dots,n\}$ let $\{P_I^{(j)}\}_{I}$ be the rough frequency projections onto intervals $I\subset\R$ in the $j$th variable. For rectangles
$R=\prod_{j=1}^n I_j$ set $P_R\coloneqq \bigotimes_{j=1}^n P_{I_j}^{(j)}$.

\begin{enumerate}[itemsep=.5em]
\item[(i)] \textbf{One parameter.}
An \emph{$\mathcal R_{2,1}$--atom} is a multiplier operator with symbol
\[
m=\sum_{I\in\mathcal I}\lambda_I \ind_I,
\]
where $\mathcal I$ is a finite family of pairwise disjoint dyadic intervals so that for each interval $I\in\mathcal I$ there exists a unique $k_I\in \Z$ such that $I\subseteq L_{k_I}$ and
\[
\sup_{k\in\Z}\ \sum_{\substack{I\in\mathcal I\\ k_I=k}}|\lambda_I|^2 \ \le\ 1.
\]
Denote by $\mathcal R_{2,1} ^{\rm at}$ the collection of $\mathcal R_{2,1}$--atoms as above. A one-parameter multiplier operator $T$ belongs to $\mathcal R_{2,1}$ if its symbol is in the atomic space generated by $\mathcal R_{2,1}$--atoms
\[
\mathcal R_{2,1}\coloneqq \left\{\sum_i a_i m_i : \,  m_i\in\mathcal R_{2,1} ^{\rm at}\,\,\forall i, \quad \sum_i |a_i|<+\infty \right\}.
\]
Define also 
\[
\|m\|_{\mathcal R_{2,1}}\coloneqq \inf\left\{\sum_i |a_i|:\, m=\sum_i a_i m_i,\quad m_i\in\mathcal R_{2,1} ^{\rm at}\,\,\forall i\right\}.
\]

\item[(ii)] \textbf{Inductive definition for $n>1$.} Assume $\mathcal R_{2,n-1}$--atoms and the norm $\|\cdot\|_{\mathcal R_{2,n-1}}$ have been defined. For $j\in\{1,\ldots,n\}$, an $\mathcal R_{2,n}^j$--atom is a multiplier operator with symbol $m$ such that $m$ admits a representation of the form
\[
m=\sum_{I\in\mathcal I}\lambda_I  \ind_I(\xi_j)  m_I(\hat \xi_j),
\]
where:
\begin{enumerate}[itemsep=.5em]
\item $\mathcal I$ is a finite pairwise disjoint family of dyadic intervals in the $j$th coordinate;
\item the coefficients satisfy the normalization
\[
\sup_{k\in\Z}\ \sum_{\substack{I\in\mathcal I\\ k_I=k }}|\lambda_I|^2 \ \le\ 1;
\]
\item for each $I\in\mathcal I$, the multiplier $m_I$ depends on the remaining $n-1$ variables and is an
$\mathcal R_{2,n-1}$--atom with respect to those variables.
\end{enumerate}
Denote by $\mathcal R_{2,n} ^{{\rm at},j}$ the class of such atoms. An $n$-parameter multiplier operator $T$ belongs to $\mathcal R_{2,n} ^j$ if its symbol is in the class generated by $\mathcal R_{2,n} ^{{\rm at},j}$--atoms
\[
\mathcal R_{2,n} ^j \coloneqq \left\{\sum_i a_i m_i: \, m_i\in\mathcal R_{2,n} ^{{\rm at},j}\,\,\forall i, \quad \sum_i |a_i|<+\infty \right\}.
\]
We equip this space with the atomic norm
\[
\|m\|_{\mathcal R_{2,n} ^j }\coloneqq \inf\left\{\sum_i |a_i|:\, m=\sum_i a_i m_i, \quad m_i\in\mathcal R_{2,n} ^{{\rm at},j}\,\,\forall i\right\} .
\]
Finally define the class of $n$-parameter $\mathcal R_2$--multipliers as 
\[
\mathcal R_{2,n} \coloneqq \bigcap_{j=1} ^n \mathcal R_{2,n} ^j,\qquad \|m\|_{\mathcal R_{2,n}}\coloneqq \max_{1\leq j \leq n} \|m\|_{\mathcal R_{2,n} ^j }.
\]
\end{enumerate}
When $n=1$ is clear from context we just write $\mathcal R_2$ in place of $\mathcal R_{2,1}$.

If $n \in \mathbb{N}$ and $\vec m \coloneqq \{ m_{\tau} \}_{\tau}$, where for each $\tau$ is an $n$-parameter $\mathcal R_2$--multiplier, we set
\[
\| \vec m \|_{\mathcal{R}_{2,n ; \tau}} \coloneqq \sup_{\tau} \| m_{\tau} \|_{\mathcal R_{2,n}}
\]
and we denote the corresponding vector-valued $\mathcal R_2$--multiplier operator by $\vec{T}_{\vec m}$.
 \end{definition}

Below we consider, as in \cite{Xu}*{p.~288}, the space of functions of bounded variation in $n$ parameters, denoted by $\mathcal V_q(\R^{\otimes n})$. The following proposition is a special case of a result from \cite{Xu}.

\begin{proposition}[\cite{Xu}*{Lemma 2}] Let $m\in\mathcal V_q(\R^{\otimes n})$, $1\leq q<2$, be an $n$-parameter multiplier of bounded $q$-variation. Then $m\in\mathcal{R}_{2,n}$. Furthermore, the corresponding multiplier norm in $\mathcal R_{2,n}$ is controlled by the $\mathcal{V}_q(\R^{\otimes n})$--norm with a constant depending only on $q$. In particular, $n$-parameter Marcinkiewicz multipliers belong to $\mathcal R_{2,n}$.
\end{proposition}

\section{Local endpoint for $n$-parameter $\mathcal R_2$--multipliers}\label{sec:proofmain} This section contains the proof of the sharp endpoint local estimate for $\mathcal R_{2,n}$--multipliers, namely the proof of Theorem~\ref{thm:multiparam}. Before turning to the proof we make some initial reductions. The simplest one is to reduce matters to the case where $R$ is the unit cube  $Q=[0,1]^n$. If $R$ has dyadic sidelengths, this follows immediately from translation invariance and dyadic dilation invariance of the desired estimate. Indeed, due to the Littlewood--Paley structure underlying the definition of $\mathcal R_{2,n}$--multipliers, the class $\mathcal R_{2,n}$ is invariant under dyadic rescalings of the frequency variables, and the normalized $L\log^{\alpha_n}L$ quasi-norms are invariant under the corresponding spatial rescalings.

For a general rectangle $R$, we may enlarge $R$ to a rectangle $R^\ast$ with the same center and dyadic sidelengths satisfying $|R|\le |R^\ast|\le C|R|$, for a positive constant $C$ that depends only on $n$. Since the left-hand side is monotone in $R$ and the right-hand side is normalized by $|R|\simeq |R^\ast|$, it suffices to prove the estimate for $R^\ast$. Thus we may assume without loss of generality that $R=Q=[0,1]^n$.

Now consider some function $f$ with $\supp f\subset Q$ as in the assumptions of the theorem. We show below that $f$ can be written as a finite sum of pieces with cancellation in one of the variables, plus a constant piece. To this end, we need some notation.  

For $j\in\{1,\dots,n\}$ we write $\mathbb E^{x_j}_0$ for the dyadic conditional expectation at scale $0$ in the $x_j$-variable, that is, averaging over dyadic intervals of length $1$ in the $x_j$-coordinate while keeping the remaining coordinates fixed.

\begin{lemma}\label{lem:onevar-cancel} Let $\{f_\tau\}_\tau $ be such that $\supp f_\tau\subset Q\coloneqq [0,1]^n$ and $f_\tau \in L^1(Q)$ for each $\tau$. Then there exist functions $g_{\tau,1},\ldots,g_{\tau,n}$ such that
\[
f_\tau=\sum_{j=1}^n g_{\tau,j}+\langle f_\tau \rangle_Q\,\ind_Q\quad\forall \tau.
\]
Furthermore, for each $j\in\{1,\dots,n\}$ and for each $\tau$, there holds
\[
\mathbb E^{x_j}_0 g_{\tau,j}=0.
\]
\end{lemma}

\begin{proof}
Set $F_{\tau,0}\coloneqq f_\tau$ and define recursively
\[
F_{\tau,j}\coloneqq \mathbb E^{x_j}_0 F_{\tau,j-1},\qquad j=1,\dots,n.
\]
Since each $f_\tau$ is supported on $Q$ and $\mathbb E^{x_j}_0$ averages over intervals of length $1$ contained in $[0,1]$, we have $\supp F_{\tau,j}\subset Q$ for every $j$. Note that $F_{\tau,n}=\langle f_\tau \rangle_Q \ind_Q $. Now set
\[
g_{\tau,j}\coloneqq F_{\tau,j-1}-\mathbb E^{x_j}_0F_{\tau, j-1},\qquad j=1,\dots,n.
\]
Since $F_{\tau,j}=\mathbb E^{x_j}_0F_{\tau,j-1}$, we have $g_{\tau,j}=F_{\tau,j-1}-F_{\tau,j}$ and
\[
f_\tau-F_{\tau,n}=\sum_{j=1}^n g_{\tau,j}.
\]
Finally, by idempotence of conditional expectation,
\[
\mathbb E^{x_j}_0 g_{\tau,j} =\mathbb E^{x_j}_0F_{\tau,j-1}-\mathbb E^{x_j}_0\mathbb E^{x_j}_0F_{\tau,j-1} =F_{\tau,j}-F_{\tau,j}=0
\] 
as desired.
\end{proof}

The next reduction is to get rid of the constant term $\{\langle f_\tau\rangle_Q \ind_Q\}_\tau$ in the decomposition of
$\vec f\coloneqq \{f_\tau\}_\tau$.
Assume first that
\[
\alpha>\left\langle \|\vec f \|_{\ell^2 _\tau}\right\rangle_Q .
\]
Then for any $L^2(\ell^2)$-bounded operator $\vec T=\{T_\tau\}$ with
$\|\vec T\|_{L^2(\ell^2)\to L^2(\ell^2)}\leq 1$, we have
\[
\left|\left\{x\in Q:\,
\|T_\tau(\langle f_\tau\rangle_Q \ind_Q )\|_{\ell^2 _\tau}>\alpha
\right\}\right|
\leq
\alpha^{-2}\sum_\tau |\langle f_\tau\rangle_Q|^2\,|Q|
\leq
\frac{|Q|}{\alpha}\left\langle\|\vec f \|_{\ell^2 _\tau}\right\rangle_Q .
\]
If on the other hand $\alpha<\langle \|\vec f \|_{\ell^2 _\tau}\rangle_Q$, then
trivially
\[
\left|\left\{x\in Q:\,
\|T_\tau(\langle f_\tau\rangle_Q \ind_Q )\|_{\ell^2 _\tau}>\alpha
\right\}\right|
\leq |Q|
\leq
\frac{|Q|}{\alpha}\left\langle\|\vec f \|_{\ell^2 _\tau}\right\rangle_Q .
\]
Thus the conclusion of the theorem holds for the constant term $\{\langle f_\tau\rangle_Q \ind_Q\}_\tau$ with the smaller $L^1$-norm on the right-hand side.

Because of the calculation above and the decomposition of Lemma~\ref{lem:onevar-cancel}, it suffices to focus on one of the cancellative pieces $\{g_{\tau,j}\}_\tau$ for some fixed index $j\in\{1,\ldots,n\}$. For this family we have the cancellation property
\begin{equation}\label{eq:cancel}
\int_0^1 g_{\tau,j}(x_1,\ldots,x_{j-1},t,x_{j+1},\ldots,x_n)\,\d t = 0
\end{equation}
for every $\tau$ and for almost every $\hat x_j\in[0,1]^{n-1}$. Moreover, by the symmetric definition of the multiparameter $\mathcal R_{2,n}$ class, we may assume that the multiplier symbol $m$ belongs to $\mathcal R_{2,n}^j$. Since the argument is independent of the choice of coordinate, we may relabel the variables and henceforth assume that $j=1$. Accordingly, we write $x=(x_1,x')\in Q=Q_1\times Q_{n-1}$, where $Q_1=[0,1]$ and $Q_{n-1}=[0,1]^{n-1}$. 

With these reductions understood, we now use the atomic definition of $\mathcal R_{2,n}^1$--multipliers together with the approximation argument in \cite{TW}*{p.~533--534}. By the latter we may assume that each $T_\tau$ is of the form
\[ 
T_\tau f= \sum_{I\in\mathcal I_\tau} \lambda_{\tau,I}\, P_I^{x}\, T^{y}_{m_{\tau,I}} f,
\]
where for each fixed $\tau$ the family $\mathcal I_\tau$ consists of dyadic intervals in the $x$-frequency variable with pointwise overlap $N_\tau$, and the coefficients satisfy the normalization
\[
\sup_{k\in\mathbb Z} \sum_{\substack{I\in\mathcal I_\tau\\ k_I=k}} |\lambda_{\tau,I}|^2 \leq \frac{1}{N_\tau}  \qquad\text{for each }\tau.
\]
Moreover, for each fixed $I$ the operator $\{T^{y}_{m_{\tau,I}}\}_\tau$ is an
$(n-1)$-parameter vector-valued $\mathcal R_{2,n-1}$ multiplier in the $y$-variables,
with $\mathcal R_{2,n-1}$--norm bounded uniformly in $I$.

\begin{proof}[Proof of Theorem~\ref{thm:multiparam}] Consider $\vec f=\{f_\tau\}_\tau$ supported in $Q$. By homogeneity we may assume
\begin{equation}\label{eq:normal}
\left\|\ \|\vec f\|_{\ell^2_\tau}\ \right\|_{L\log^{\alpha_n}L(Q)}=1.
\end{equation}
Furthermore, by the discussion at the beginning of this section, we may assume that $\vec f$ satisfies the cancellation condition~\eqref{eq:cancel}. It thus suffices to prove
\[
\left|\left\{ u \in Q:\, \left \|\,\vec T\vec f(u)\, \right\|_{\ell^2 _\tau}>\alpha \right\}\right|\lesssim\frac{1}{\alpha}
\]
for $\vec f=\{f_\tau\}_\tau$ supported on $Q$ and satisfying \eqref{eq:cancel} and  \eqref{eq:normal}.

 The proof is by induction on $n$. The base case $n=1$ is the content of Theorem~\ref{thm:vvaluedR_2}. Assume now that the theorem has been proved for
$n-1$. In particular, if $\vec S=\{S_\tau\}$ is an $(n-1)$-parameter vector-valued $\mathcal R_{2,n-1}$--multiplier acting on the $y$-variables, then the inductive hypothesis, together with the standard weak-to-strong upgrade on $Q_{n-1}$ (see e.g. \cite{Bak2019}*{Lemma 3.2}), implies
\begin{equation}\label{eq:induct}
\left\| \ \left\|\,\vec S(\vec h)\, \right\|_{\ell^2 _\tau}\right\|_{L^1(Q_{n-1})}
\lesssim
\left\|\,\vec h\, \right\|_{L\log ^{\alpha_{n-1}+1} L(Q_{n-1};\ell^2 _\tau)}
\end{equation}
for all $\ell^2_\tau$-valued functions $\vec h$ supported in the unit cube
$Q_{n-1}\subset\R^{n-1}$.

For each $\tau$ and $I\in\mathcal I_\tau$ we insert the auxiliary cutoffs $S_{k_I}^x$ and $S_I^x$; recall that these are smooth frequency projections with symbol identically $1$  on $L_{k_I}$ and $I$, and supported on $2L_{k_I}$ and $2I$, respectively. For each fixed $\tau$ and $I\in\mathcal I_\tau$ we write
\[
\lambda_{\tau,I} P_I^x T^y_{m_{\tau,I}} f_\tau = P_I^x\, S_{k_I}^x\, S_I^x\left(\lambda_{\tau,I} T^y_{m_{\tau,I}} f_\tau\right).
\]
It will be convenient to split for each $\tau$,
\[
 T_\tau  f_\tau = \sum_{\substack{I\in\mathcal I_\tau \\ k_I<0}} \lambda_{\tau,I}\, P_I^{x}\, T^{y}_{m_{\tau,I}} f_\tau+\sum_{\substack{I\in\mathcal I_\tau \\ k_I\geq 0}} \lambda_{\tau,I}\, P_I^{x}\, T^{y}_{m_{\tau,I}} f_\tau \eqqcolon T_{\tau,<} f_\tau + T_{\tau,>} f_\tau .
\]
Let 
\[
\vec T_>\coloneqq \{T_{\tau,>}\}_\tau,\qquad 	\vec T_<\coloneqq \{T_{\tau,<}\}_\tau,\qquad \vec T=\vec T_{<}+\vec T_>.
\]

\subsection*{The case $k_I<0$ and the proof for $\vec T_<$} If $I\in\mathcal I_\tau$ is such that $k_I<0$ we set
\[ 
f_{\tau,I}\coloneqq S_I ^x S_{k_I} ^x \lambda_{\tau,I} T_{m_{\tau,I}} ^yf_\tau,\qquad | f_{\tau,I}|\lesssim\phi_{|I|^{-1}}\ast^{x}  | T_{m_{\tau,I}} ^y (\lambda_{\tau,I}\omega_{2^{-k_I}}\ast^x f_\tau)|\coloneqq\phi_{|I|^{-1}}\ast^{x} F_{\tau,I}.
\]
Explicitly,
\[
\vec T_< \vec f= \sum_{\substack{I\in\mathcal I_\tau \\ k_I<0}} \lambda_{\tau,I}\, P_I^{x}\, T^{y}_{m_{\tau,I}} f_\tau=\sum_{\substack{I\in\mathcal I_\tau \\ k_I<0}}  P_I^{x} f_{\tau,I}.
\]
Hence, an application of Proposition~\ref{prp:vvaluedTW} in the $x$-variable yields
\[
\begin{split}
&  \left|\left\{(x,y)\in Q_1\times Q_{n-1}:\, \left\|\vec T_< \vec f\right\|_{\ell^2 _\tau}>\alpha\right\}\right|\lesssim \frac{1}{\alpha} \int_{Q_{n-1}}\int_{\R} \left(\sum_{\tau} N_\tau \sum_{I\in\mathcal I_\tau} |F_{\tau,I}|^2\right)^{\frac{1}{2}}\,\d x\,\d y
\\
&\qquad = \frac{1}{\alpha} \int_{Q_{n-1}}\int_{\R}  \left(\sum_{\tau,I}N_\tau|T_{m_{\tau,I}} ^y (\lambda_{\tau,I}\omega_{2^{-k_I}}\ast^{x}f_\tau)|^2\right)^{\frac{1}{2}}\,\d x\,\d y
\\
&\qquad \lesssim \frac{1}{\alpha}  \int_{\R} \left\|\left(\sum_{\tau,I}N_\tau| \lambda_{\tau,I}(\omega_{2^{-k_I}}\ast^x f_\tau)|^2\right)^{\frac{1}{2}}  \right\|_{L\log^{\alpha_{n-1}+1}L(Q_{n-1},\d y)}\, \d x,
\end{split}
\]
where the last estimate follows by the inductive assumption in the form of \eqref{eq:induct}. Now note that
\[
\begin{split}
 \left(\sum_{\tau,I}N_\tau| \lambda_{\tau,I}(\omega_{2^{-k_I}}\ast^x f_\tau)|^2\right)^{\frac{1}{2}}&\leq
 \left(\sum_\tau N_\tau \sum_{k<0} \sum_{k_I=k} | \lambda_{\tau,I}|^2| (\omega_{2^{-k}}\ast^x f_\tau)|^2\right)^{\frac{1}{2}}
 \\
& \leq \left(\sum_\tau   \sum_{k<0} | (\omega_{2^{-k}}\ast^x f_\tau)|^2\right)^{\frac{1}{2}}
\eqqcolon G.
\end{split}
\]
Here we used the normalization of the coefficients $\{\lambda_{\tau,I}:\, k_I=k\}$ for each $k<0$. Thus it suffices to estimate
\begin{equation}\label{eq:orl}
\int_{\R} \left\|G \right\|_{L\log^{\alpha_{n-1}+1}L(Q_{n-1},\d y)} \, \d x \lesssim 1.
\end{equation}
We proceed with the proof of \eqref{eq:orl}. Using the cancellation assumption \eqref{eq:cancel} and writing $F(x,y)\coloneqq \|\vec f(x,y)\|_{\ell^2_\tau}$ we have for each $k<0$,
\[
|\omega_{2^{-k}}\ast^x f_\tau(x,y)| \lesssim 2^{2k}\,\widetilde \omega_{2^{-k}}\ast^x |f_\tau|(x,y)
\]
for some fast decaying bump function $\widetilde \omega$. Taking the $\ell^2$-norm in $\tau$ and using Minkowski's inequality yields 
\[
\left(\sum_\tau |\omega_{2^{-k}}\ast^x f_\tau(x,y)|^2\right)^{1/2} \lesssim 2^{2k}\,\widetilde \omega_{2^{-k}}\ast^x F(x,y).
\]
Bounding the $\ell^2 _k$ norm by the $\ell^1 _k$ norm, we have
\[
G(x,y) \lesssim \sum_{k<0}2^{2k}\,\widetilde \omega_{2^{-k}}\ast^x F(x,y) = K\ast^x F(x,y),
\]
where
\[ K  \eqqcolon  \sum_{k<0}2^{2k}\,\widetilde \omega_{2^{-k}}. 
\]
Note that $K\in L^1(\R)$ with $\|K\|_{L^1(\R)}\lesssim 1$. Set $r_n \coloneqq \alpha_{n-1}+1=\alpha_n-1/2$. Using Minkowski's inequality for the Orlicz space $L\log^{r_n} L([0,1])$, yields
\begin{align*}
\left\|G(x,\cdot)\right\|_{L\log^{r_n} L(Q_{n-1},\d y)}&\lesssim \left\|(K\ast^x F)(x,\cdot)\right\|_{L\log^{r_n} L(Q_{n-1},\d y)}
\\
&\le
\int_{\mathbb R} |K(t)| \left\|F(x-t,\cdot)\right\|_{L\log^{r_n} L(Q_{n-1},\d y)}\,\d t \\
&=
\left(|K|\ast^x \|F(x,\cdot)\|_{L\log^{r_n} L(Q_{n-1},\d y)}\right)(x).
\end{align*}
In particular,
\[
\begin{split}
&\int_{\R}\left\|G(x,\cdot)\right\|_{L\log^{r_n} L(Q_{n-1},\d y)}\, \d x \lesssim \int_{\R} \|F(x,\cdot)\|_{L\log^{r_n} L(Q_{n-1},\d y)} \,\d x
\\
& \qquad \lesssim \int_{Q_1} \left[1+\int_{Q_{n-1}} \Phi_{r_n}\left(|F(x,y)|\right)\, \d y \right]\,\d x 
\\
& \qquad = 1+\int_Q \Phi_{r_n}\left(|F(x,y)|\right)\,\d x \, \d y \lesssim 1+\Phi_{\alpha_n-\frac{1}{2}}\left(\|F\|_{L\log^{\alpha_n-\frac{1}{2}}L(Q)}\right),
\end{split}
\]
where we used (ii) and (iii) of Proposition~\ref{prp:normequiv} in the last line. Remembering our normalization \eqref{eq:normal} the right-hand side of the display above is bounded by $1+\Phi_{\alpha_n-\frac{1}{2}}\left(1\right)\lesssim 1$ and the proof of the inductive step for $\vec T_{<}$ is complete.

\subsection*{The case $k_I\geq 0$ and the proof for $\vec T_>$}

Let us denote by $\mathbb E_k ^x f$ the conditional expectation and by $\mathbb D_k ^x f$ the martingale differences, as operators acting on the $x$-variable. By the Chang--Wilson--Wolff inequality in one parameter, but applied to functions of $n$ variables, we have
\[
\left\| \| \{f_\tau-\mathbb E_0 ^x f_\tau\}\|_{\ell^2 _\tau}\right\|_{L^p(Q)}\lesssim p^{\frac{1}{2}} \left\| \left(\sum_{\tau} \sum_{\ell\geq 0}| \mathbb D_\ell ^x f_\tau|^2\right)^{\frac{1}{2}} \right\|_{L^p(Q)},
\]
with implicit constant independent of $p$. By repeating verbatim the proof of \cite{BCPV}*{Theorem D} we deduce the existence of functions $f_{\tau,\ell}$ such that for all $\tau$ and nonnegative integers $\ell$ there holds
\begin{equation}\label{eq:weaksfvv}
\begin{split}
&\qquad \quad\mathbb D_\ell ^x f_\tau =\mathbb D_\ell ^x f_{\tau,\ell} \quad  \forall \ell\geq 0,\qquad \supp f_{\tau,\ell} \subset 4Q_1,\\
& \qquad \left\|\left(\sum_{\tau} \sum_{ \ell\geq 0}|f_{\tau,\ell}|^2 \right)^{\frac{1}{2}}\right\|_{L\log^{\alpha_n-\frac{1}{2}}L(Q)}\lesssim \left\|\left(\sum_{\tau}|f_{\tau}|^2 \right)^{\frac{1}{2}}\right\|_{L\log^{\alpha_n}L(Q)}= 1,
\end{split}
\end{equation}
since $\vec f$ satisfies the normalization~\eqref{eq:normal}.

For any function $h$ on $Q$ and $\theta \in [-1/3,1/3]$ let us write $h^\theta(x,y)\coloneqq h(x-\theta,y)$. Now fix $(x,y)\in Q_1\times Q_{n-1}$. Write $g^{\theta} _{\tau,\ell}$ for the functions generated by applying \eqref{eq:weaksfvv} to the vector $\{f^\theta _\tau\}_\tau$. A perusal of the proof of \cite{TW}*{Proposition 4.1} allows us to write for each nonnegative integer $k$, for every $\tau$, and any Fourier multiplier operator $T^y$ acting on the $y$-variable, 
\[
\begin{split}
|S_k ^x (T^y  f_\tau)|  &= |T^y(S_k ^x f_\tau)| =\left|T^y\left( \sum_{\ell\geq 0}  S_k ^x ( \mathbb D_\ell ^x f_\tau)\right)\right| 
=\left|T^y\left(\sum_{\ell\geq 0}  S_k ^x (\mathbb D_\ell ^x f_\tau ^\theta)(\cdot+\theta)\right)\right| \\
&=\left|T^y\sum_{\ell\geq 0}\sum_{|J|=2^{-\ell}}  \left( \langle g_{\tau,\ell} ^\theta ,h_J\rangle S_k ^x  h_J(x+\theta)\right)\right|
\\
&\leq \sum_{\ell\geq 0}\sum_{|J|=2^{-\ell}} \left(\int_J |T^y(g_{\tau,\ell} ^\theta)|\right)|J|^{-1/2}|S_kh_J(x+\theta)|.
\end{split}
\]
Now, for each $x,\ell$ there exists a measurable set $\Theta(x,\ell)\subset [-1/3,1/3]$ with $|\Theta(x,\ell)|\simeq 1$ such that, for every  $\theta\in \Theta(x,\ell)$, there holds
\[
\sum_{|J|=2^{-\ell}} \left(\int_J |T^y(g_{\tau,\ell} ^\theta)|\right)|J|^{-1/2}|S_k ^x h_J(x+\theta)|\lesssim  2^{-|\ell-k|/2} |(T^y g_{\tau,\ell} ^\theta)|\ast^x \phi_{2^{-k}}.
\]
Thus, averaging over $\theta\in \Theta(x,\ell)$,
\[
|S^x _k T^y f_\tau|\lesssim \sum_{\ell \geq 0}2^{-|\ell-k|/2} \langle  |\phi_{2^{-k}}\ast^x T^y(g_{\tau,\ell} ^\theta)| \rangle_{\Theta(x,\ell)} \lesssim \phi_{2^{-k}}\ast^x \sum_{\ell \geq 0}2^{-|\ell-k|/2} \int_{[-\frac{1}{3}, \frac{1}{3} ]} |T^y(g_{\tau,\ell} ^\theta)| \, \d \theta.
\]
We have proved
\[
 |S ^x _k T^y f_\tau| \lesssim \phi_{2^{-k}}\ast^x \left(\sum_{\ell \geq 0}2^{-|\ell-k|/2} \int_{[-\frac{1}{3}, \frac{1}{3} ]} |T^y(g_{\tau,\ell} ^\theta)| \, \d \theta\right),\qquad k\geq 0.
\]
Applying this with $T_{m_{\tau,I}} ^y$ in place of $T^y$ for $k_I\geq 0$, we have for each $y\in [0,1]^{n-1}$ fixed
\[ 
 |S_{k_I} ^x [(T_{m_{\tau,I}} ^y f_\tau)(\cdot,y)]|\lesssim \phi_{2^{-k_I}}\ast^x F_{\tau,I}(\cdot,y),\qquad F_{\tau,I} \coloneqq \sum_{\ell \geq 0}2^{-|\ell-k_I|/2} \int_{[-\frac{1}{3}, \frac{1}{3} ]} |T_{m_{\tau,I}} ^y(g_{\tau,\ell} ^\theta)| \, \d \theta.
\]
As before, we write
\begin{equation}\label{eq:vvassumption}
\begin{split}
P_I ^x T_{m_{\tau,I}} ^y f_\tau & = P_I ^x (S_I ^x S_{k_I} ^x T_{m_{\tau,I}} ^y f_\tau) \eqqcolon P_I ^x f_{\tau,I},\qquad k_I\geq 0,
\\[1em]
|f_{\tau,I}(\cdot,y)| & \lesssim \omega_{|I|^{-1}}\ast^x \phi_{2^{-k_I}}\ast^x F_{\tau,I}(\cdot,y) \lesssim \widetilde \omega_{|I|^{-1}} \ast^x F_{\tau,I}(\cdot,y),
\end{split}
\end{equation}
since we always have $2^{-k_I}<|I|^{-1}$.

Inserting \eqref{eq:vvassumption} proved above as input in Proposition~\ref{prp:vvaluedTW}, applied in the $x$-variable, we  have
\[
\begin{split} 
& \left|\left\{(x,y)\in Q_1 \times Q_{n-1}:\, \left\|\vec T_> \vec f(x,y)\right\|_{\ell^2 _\tau}>\alpha \right\}\right| \lesssim \frac{1}{\alpha} \int_{Q_{n-1}} \int_\R \left(\sum_\tau N_\tau \sum_{I} |\lambda_{\tau,I} F_{\tau,I}(x,y)|^2\right)^{\frac{1}{2}}\, \d x \, \d y
\\
&\qquad  \lesssim \frac{1}{\alpha} \int_{[-\frac{1}{3}, \frac{1}{3}]} \int_{Q_{n-1} \times \R} \left(\sum_\tau N_\tau \sum_I \left|\lambda_{\tau,I} \sum_{\ell \geq 0}2^{-|\ell-k_I|/2}|T_{m_{\tau,I}} ^y(g_{\tau,\ell} ^\theta)|  \right|^2 \right)^{\frac{1}{2}} \d x \,\d y \, \d \theta
\\
&\qquad  \leq \frac{1}{\alpha} \int_{[-\frac{1}{3}, \frac{1}{3}]}  \int_{Q_{n-1}\times \R} \left(\sum_\tau N_\tau \sum_I |\lambda_{\tau,I}|^2 \left( \sum_{\ell\geq 0} 2^{-|\ell-k_I|/2}\right) \sum_{\ell \geq 0}2^{-|\ell-k_I|/2}|T_{m_{\tau,I}} ^y(g_{\tau,\ell} ^\theta)| ^2  \right)^{\frac{1}{2}} \d x \,\d y \, \d \theta
\\
&\qquad  \lesssim \frac{1}{\alpha} \int_{[-\frac{1}{3}, \frac{1}{3}]} \int_{\R} \int_{Q_{n-1}} \left(\sum_{\tau}N_\tau \sum_I |\lambda_{\tau,I}|^2  \sum_{\ell \geq 0}2^{-|\ell-k_I|/2}|T_{m_{\tau,I}} ^y(g_{\tau,\ell} ^\theta)| ^2  \right)^{\frac{1}{2}} \d y \, \d x\, \d \theta.
\end{split}
\]
Fix $\theta\in[-1/3,1/3]$ and for each $(\tau,I,\ell)$ with $k_I\geq 0$, set
\[
 T_{\tau,I,\ell}\coloneqq T_{m_{\tau,I}} ^y,\qquad h_{\tau,I,\ell}\coloneqq N_\tau ^{\frac{1}{2}}\lambda_{\tau,I} 2^{-|\ell-k_I|/4}g_{\tau,\ell} ^\theta.
\]
Using the inductive assumption~\eqref{eq:induct}, we have
\[
\left\| \left(\sum_{\tau} \sum_{I} \sum_{\ell \ge 0} |T_{\tau,I,\ell}(h_{\tau,I,\ell})|^2\right)^{\frac{1}{2}}\right\|_{L^1(Q_{n-1},\d y)} \lesssim \left\| \left(\sum_{\tau} \sum_{I} \sum_{\ell \ge 0} |h_{\tau,I,\ell}|^2\right)^{\frac{1}{2}}\right\|_{L\log ^{\alpha_{n-1}+1} L (Q_{n-1},\d y)}.
\]
Using the normalization of the coefficients $\{\lambda_{\tau,I}:\, k_I=k\}$, we now estimate
\[
 \left(\sum_{\tau} \sum_I \sum_{\ell \ge 0} | h_{\tau,I,\ell} |^2\right)^{\frac{1}{2}}=\left(\sum_\tau N_\tau \sum_{k\geq 0} \sum_{\ell\geq 0} \sum_{k_I=k} |\lambda_{\tau,I}|^2 2^{-|k-\ell|/2} |g_{\tau,\ell} ^\theta|^2 \right)^{\frac{1}{2}}\lesssim  \left(\sum_{\tau} \sum_{\ell\geq 0} |g_{\tau,\ell} ^\theta|^2\right)^{\frac{1}{2}}.
\]
Combining the estimates above and using Fubini, we get
\begin{multline*}  
\left|\left\{(x,y)\in Q_1 \times Q_{n-1}:\, \|\vec T_> \vec f(x,y)\|_{\ell^2_\tau}>\alpha \right\}\right| \lesssim \\
 \int_{[-\frac{1}{3}, \frac{1}{3} ]} \frac{1}{\alpha}\, \int_{4Q_1} \left\|\left(\sum_{\tau} \sum_{\ell\geq 0} |g_{\tau,\ell} ^\theta|^2\right)^{\frac{1}{2}}\right\|_{L\log ^{\alpha_{n-1}+1}L(Q_{n-1},\d y)} \d x\,\d \theta,
\end{multline*}
since $\supp g_{\tau,\ell} ^\theta \subset 4Q_1$ for each $\tau$ and $\ell\geq 0$. Note that $\alpha_{n-1}+1=\alpha_n-1/2$. Let us call 
\[
F\coloneqq \left( \sum_{\tau} \sum_{ \ell\geq 0} |g_{\tau,\ell} ^\theta|^2\right)^{\frac{1}{2}}
\]
in order to simplify notation. Using (ii) of Proposition~\ref{prp:normequiv} together with \eqref{eq:weaksfvv}, we estimate as follows
\[
\begin{split} 
 &\int_{4Q_1}\left\|\left(\sum_{\tau} \sum_{\ell \geq 0} |g_{\tau, \ell} ^\theta|^2\right)^{\frac{1}{2}}\right\|_{L\log ^{\alpha_{n-1}+1}L(Q_{n-1},\d y)}  \lesssim \int_{4Q_1} \left(1+\int_{Q_{n-1}} \Phi_{\alpha_n-1/2}(F) \,\d y\right) \, \d x 
\\
&\qquad \lesssim 1+ \int_{4Q}\Phi_{\alpha_n-1/2}(|F|) \,\d x\, \d y   \lesssim 1+ \Phi_{\alpha_n-1/2}\left(\|F\|_{L\log^{\alpha_n-1/2}L(Q)}\right) 
\\
&\qquad \lesssim  1+ \Phi_{\alpha_n-\frac{1}{2}}\left( \left\| \|\vec f^\theta \|_{\ell^2 _\tau} \right\|_{L\log ^{\alpha_n}L(Q)}\right)  \leq 1 +\Phi_{\alpha_n-\frac{1}{2}}(1)
\end{split}
\]
by the normalization~\eqref{eq:normal}, which is preserved under the translation defining $\vec f^\theta$. Integrating in $\theta\in[-1/3,1/3]$ yields
\[
\left|\left\{(x,y)\in Q:\, \left\|\vec T_> \vec f\right\|_{\ell^2 _\tau}>\alpha \right\}\right|
\lesssim
\int_{[-\frac{1}{3},\frac{1}{3}]}
\frac{1}{\alpha}
\left(1+\Phi_{\alpha_n-\frac{1}{2}}(1)\right)\,\d \theta
\simeq \frac{1}{\alpha}.
\]
This completes the proof for $\vec T_>$ and with that, the proof of the theorem.
\end{proof}

\section{Proof of Theorem~\ref{thm:Lp_growth}}\label{sec:Proof_growth}  Let $n \in \N$. The case $n=1$ follows either by Theorem~\ref{thm:vvaluedR_2} and interpolation, or by the argument that follows, specialized to the case $n=1$.  For $n>1$ we use induction. As in Definition~\ref{def:R2n_atoms_norm} we write $\vec T$ as a convex combination of atomic operators of the form
\[
 \vec T: \{f_\tau\}_\tau \mapsto \sum_{I\in\mathcal I_\tau }   \lambda_{\tau,I} P_I ^x T^y _{m_{\tau,I}} f_\tau,
\]
where each $\mathcal I_\tau$ is a pairwise disjoint collection of intervals and 
\[ 
\sup_{k\in\Z} \sup_\tau \sum_{I\in\mathcal I_\tau:\, k_I=k}|\lambda_{\tau,I}|^2 \leq 1,
\]
and each $T_{m_{\tau,I}}$ is an $\ell^2 _\tau$-valued $\mathcal R_{2,n-1}$--multiplier, uniformly in $I$. It will suffice to prove the $L^p$-norm bound for each such atom.  

Given a collection $\mathcal I$, consisting of pairwise disjoint intervals on the real line, Rubio de Francia has proved in \cite{RdF1985} the square function estimate
\[
\left\|\left(\sum_I|P_If|^2\right)^{\frac{1}{2}}\right\|_{L^p(\R)}\lesssim \|f\|_{L^p(\R)},\qquad 2\leq p <\infty.
\]
The inverse Rubio de Francia square function estimate with bounds independent of $p \in (1,2)$ was proved by Bourgain on the torus, \cite{bourgainLP} (see \cite{Xu_22} for the Euclidean case), as a consequence of his extension of Rubio de Francia's inequality to $p=1$ \cite{bourg_sq}, while Kislyakov and Parilov \cite{KislPar} extended Bourgain's result \cite{bourg_sq} to the range $0<p<1$. It follows from the work of Kislyakov and Parilov \cite{KislPar}, cf. \cite{KislPar}*{p.~6419} and \cite{KislMarc}*{Proposition 2}, that the following vector-valued version of the inverse Rubio de Francia square function estimate holds
\[
\left\|\left (\sum_\tau\left|\sum_{I\in\mathcal I_\tau }P_I f_\tau \right|^2\right)^{\frac{1}{2}}\right\|_{L^p(\R)}\lesssim \left\|\left( \sum_{\tau} \sum_{I \in \mathcal{I}_{\tau}} |P_I f_{\tau} |^2\right)^{\frac{1}{2}}\right\|_{L^p(\R)},\qquad 0<p \leq 2,
\] 
where, crucially, the implicit constant does not depend on $p$. Applying this estimate to the collections $\{ \mathcal I_\tau\}_\tau$ given in the definition of an $\mathcal R_{2,n}$--atom, we have
\[
\left\| \vec T (\{f_\tau\}_\tau) \right\|_{L^p(\R^n;\ell^2 _\tau)}\lesssim \left\| \left(\sum_{\tau} \sum_{I\in\mathcal I_\tau} |\lambda_{\tau,I}P_I ^x T_{m_{\tau,I}} ^y f_\tau |^2\right)^{\frac{1}{2}}\right\|_{L^p(\R^n)}.
\]
The right-hand side of the display above is bounded by the inductive assumption by a constant multiple of
\[
\max(p,p') ^{\beta_{n-1}}\left\| \left(\sum_{\tau} \sum_{I\in\mathcal I_\tau} |\lambda_{\tau,I} P_I ^x f_\tau |^2\right)^{\frac{1}{2}}\right\|_{L^p(\R^n)}.
\]

Let $\{\varepsilon_{\tau,I}\}$ be independent Rademacher signs. By the Hilbert space Khintchine inequality, see e.g. \cite{Pisier_MiBS}*{p. 154}, we have for every $0<p\leq 2$,
\[
\left\| \left(\sum_{\tau} \sum_{I\in\mathcal I_\tau} |\lambda_{\tau,I} P_I ^x f_\tau |^2\right)^{\frac{1}{2}}\right\|_{L^p(\R,\, \d x)} \simeq 
\left(\mathbb E_\varepsilon \left\|\left(\sum_{\tau}\left|\sum_{I\in\mathcal I_\tau}\varepsilon_{\tau,I}\lambda_{\tau,I} P_I ^x f_\tau\right|^2\right)^{1/2}\right\|_{L^p(\R,\d x)}^p\right)^{1/p};
\]
we stress that the implicit constants are independent of $p$ in the range considered here.
Note that for each fixed sign-choice and almost every fixed $y\in \R^{n-1}$, the operator 
\[
\{f_\tau(\cdot,y)\}_\tau \mapsto \sum_{I\in\mathcal I_\tau}\varepsilon_{\tau,I}\lambda_{\tau,I} P_I f_\tau(\cdot,y)
\]
is a vector-valued $\mathcal R_{2,1}$--multiplier, hence by the case $n=1$ of the theorem it has $L^p$-bound $\sim \max(p,p')^{\frac{3}{2}}$. We have proved
\[
\beta_n \le \beta_{n-1} +{\frac{3}{2}}
\]
and hence the proof is complete.

\bibliography{marcbib}{}
\bibliographystyle{amsra}

\end{document}